\documentclass{article}

\usepackage{amsfonts}
\usepackage{amsmath}
\usepackage{amssymb}
\usepackage{graphics}
\usepackage{graphicx}
\usepackage{color}
\usepackage{float}
\usepackage{epsfig}
\usepackage{geometry}
\usepackage{indentfirst,latexsym,bm}
\usepackage{lscape}
\usepackage{tabularx}
\usepackage{enumerate}
\usepackage{longtable}
\usepackage{lscape}
\usepackage{tikz}
\usetikzlibrary{matrix}

\numberwithin{equation}{section}

\newtheorem{theorem}{Theorem}[section]

\newtheorem{assumption}[theorem]{Assumption}

\newtheorem{corollary}[theorem]{Corollary}

\newtheorem{lemma}[theorem]{Lemma}

\newtheorem{proposition}[theorem]{Proposition}
\newtheorem{remark}[theorem]{Remark}

\newenvironment{proof}[1][Proof]{\textbf{#1.} }{\ \rule{0.5em}{0.5em}}

\def\sube{\hat{\mathbb{E}}}
\def\rd{\mathbb{R}^d}

\title{A monotone scheme for G-equations with application to the explicit convergence rate of robust central limit theorem
\thanks{Research partially supported by National Natural Science Foundation of China (No. 12171169) and Laboratory of Mathematics for Nonlinear Science, Fudan University.
The authors thank the referee, and
especially Mingshang Hu, Lianzi Jiang, Shige Peng and Yongsheng Song for their
helpful comments on the manuscript.}}
\author{Shuo Huang\thanks{Department of Statistics, The University of Warwick, Coventry CV4
7AL, U.K. \texttt{s.huang.13@warwick.ac.uk}}\and Gechun Liang\thanks{Department of Statistics, The University of Warwick, Coventry CV4
7AL, U.K. \texttt{g.liang@warwick.ac.uk} }}

\date{}

\begin{document}
\maketitle
\begin{abstract}
We propose a monotone approximation scheme for a class of fully nonlinear PDEs called G-equations.
Such equations arise often in the characterization of G-distributed random variables in
a sublinear expectation space. The proposed scheme is constructed recursively based on a piecewise constant approximation
of the viscosity solution to the G-equation. We establish the convergence of the
scheme and determine the convergence rate with an explicit error bound, using the comparison principles for both the scheme and the equation together with a mollification procedure.
The first application is
obtaining the convergence rate of Peng's robust central limit theorem with an explicit bound of Berry-Esseen type. The second application is an approximation scheme with its convergence rate for the Black-Scholes-Barenblatt equation.
\\


\noindent\textit{MSC2010 subject classifications}: 60F05, 60H30, 65M15.

\end{abstract}


\section{Introduction}

The theory of G-expectations (see \cite{Peng2008,Peng2008(2),Peng2010,Peng2010(2)}) is a natural generalization of classical probability theory in the presence of Knightian uncertainty. That is, random outcomes are evaluated, not using a single probability measure, but using the supremum over a range of possibly mutually singular probably measures. One of the fundamental results in the theory is the celebrated central limit theorem, dubbed as \emph{robust central limit theorem} by Peng in \cite{Peng2010}. It  provides a theoretical foundation for the widely used G-distributed random variables in nonlinear probability and statistics. The theorem was first proved in \cite{Peng2008} by applying the regularity theory of fully nonlinear PDEs (see \cite{Krylov3} and \cite{Wang}) to G-equations, the latter of which characterize G-distributed random variables. However, no convergence rate was derived in \cite{Peng2008}. The corresponding convergence rate was subsequently obtained
in \cite{Song} and \cite{FPSS} using Stein's method and more recently in \cite{Krylov4} using stochastic control method under different model assumptions. However, an explicit formula for the constant appearing in the convergence rate is still lacking.

In this paper, we build a monotone approximation scheme for the
G-equation, and determine its convergence rate by obtaining an
\emph{explicit} error bound between the approximate solution and
the viscosity solution of the G-equation.  \emph{This will in turn, for the first time, provide the convergence rate for Peng's robust central limit
theorem with an explicit bound of Berry-Esseen type.} The new
convergence {rate} improves all the existing ones obtained under
different model assumptions in the literature. Moreover, different
from \cite{FPSS}, \cite{Krylov4} and \cite{Song}, our method is
analytical and is developed under the framework of the monotone
approximation schemes for viscosity solutions. Thus, it unveils an
intrinsic connection between the convergence analysis of numerical
schemes in PDEs and the central limit theorem in probability. It
also introduces new tools from the numerical analysis for viscosity
solutions to the study of G-expectations {and} especially its robust
central limit theorem.\\

Let's first introduce Peng's G-equation. Let $(\Omega, \mathcal{H},
\sube)$ be a sublinear expectation space, supporting two
$d$-dimensional random vectors $X$ and $Y$. Recall that $\sube$ is a
sublinear expectation if it satisfies \emph{monotonicity, constant
preserving, sub-additivity} and \emph{positive homogeneity}
properties (see Chapter 1 in \cite{Peng2010} or
(\ref{axiom1})-(\ref{axiom4}) in section \ref{sec-scheme} for
further details). With the random vectors $(X,Y)$ and the sublinear
expectation $\sube$, we introduce the nonlinear function
$G:\rd\times\mathbb{S}(d)\to\mathbb{R}$ as
\begin{equation}\label{functionG}
G(p,A):=\sube\left[\langle p,Y\rangle+\frac12\langle AX,X\rangle\right],
\end{equation}
for $(p,A)\in\rd\times\mathbb{S}(d)$, where $\mathbb{S}(d)$ is the collection of all $d$-dimensional symmetric matrixes.

For $T\ge 1$, let $Q_T:=(0,T]\times\rd$. We consider the fully-nonlinear parabolic PDE defined on the parabolic domain $Q_T$,
\begin{equation}\label{PDE_1}
\partial_tu-G(D_xu,D_{x}^2u)=0,
\end{equation}
with initial condition
\begin{equation}\label{cc}
u|_{t=0}=\phi.
\end{equation}
In \cite{Peng2008,Peng2008(2),Peng2010,Peng2010(2)}, the PDE (\ref{PDE_1}) is referred {to} as the G-equation, which is used to characterize G-distribution. More specifically, let $(\xi, \zeta)$ be a pair of G-distributed $d$-dimensional random vectors {characterized by (\ref{PDE_1})} under another sublinear expectation $\widetilde{\mathbb{E}}$ (possibly different from $\sube$). That is, the G-distributed random vectors $(\xi,\zeta)$ satisfies
\begin{equation}\label{Gdistributed}
\widetilde{\mathbb{E}}\left[\langle p,\zeta\rangle+\frac12\langle A\xi,\xi\rangle\right]=G(p,A)
\end{equation}
for $(p,A)\in\rd\times\mathbb{S}(d)$, and
for $a,b\in\mathbb{R}^d$ and $(\tilde{\xi},\tilde{\zeta})$ as an independent copy of $(\xi, \zeta)$, the following equality holds in distribution sense:
$$\left(a\xi+b\tilde{\xi},a^2\zeta+b^2\tilde{\zeta}\right)\stackrel{d}{=}
\left(\sqrt{a^2+b^2}\xi,(a^2+b^2){\zeta}\right).$$
Moreover, $\zeta$ is called maximal distributed in the sense that there exists a bounded, closed and convex subset $Q\subset \mathbb{R}^d$ such that
\begin{equation}\label{maximal}
\widetilde{\mathbb{E}}[\psi(\zeta)]=\max_{q\in Q}\psi(q),
\end{equation}
for any continuous function $\psi$ satisfying linear growth condition.
Note that the existence of $(\xi, \zeta)$ is guaranteed by Proposition 4.2 of \cite{Peng2008}.
Then, it has been proved in Proposition 4.8 of \cite{Peng2008} that
(\ref{PDE_1})-(\ref{cc}) admits a unique viscosity solution $u$ which admits the representation
\begin{equation}\label{solution}
u(t,x)=\widetilde{\mathbb{E}}[\phi(x+\sqrt{t}\xi+t\zeta)],
\end{equation}
provided that the initial data $\phi$ satisfies some regularity condition. However, it is not clear how to explicitly solve (\ref{PDE_1})-(\ref{cc}) in order to
characterize the G-distributed random vectors $(\xi, \zeta)$ except for some special cases, so a numerical scheme for (\ref{PDE_1})-(\ref{cc}) is needed.\\

In this paper, we propose a numerical scheme to approximate the viscosity solution $u$ of
(\ref{PDE_1})-(\ref{cc}) by merely using the random vectors $(X,Y)$ under $\sube$ as input. Note that $(X,Y)$ could  follow arbitrary distributions (one example is shown in section \ref{section: BSB_equ}). For $\Delta\in(0,1)$, we introduce $u^\Delta: [0,T]\times\rd\to\mathbb{R}$ recursively as
\begin{equation}\label{scheme}
u^\Delta(t,x)=\sube[u^\Delta(t-\Delta,x+\sqrt{\Delta}X+\Delta Y)]\mathbf{1}_{\{t\ge \Delta\}}+\phi(x)\mathbf{1}_{\{t<\Delta\}}.
\end{equation}
The above recursive approximation implies that, for any $n\in\mathbb{N}$ such that $n\Delta\le T$ and $t\in[n\Delta,((n+1)\Delta)\wedge T)$, $u^{\Delta}(t,\cdot)$ is a constant in $t$ and is given by
\begin{equation}\label{interp}
u^\Delta(t,\cdot)\equiv u^\Delta(n\Delta,\cdot),
\end{equation}
and at time $n\Delta$, there is a jump of the size
$$u^\Delta(n\Delta,\cdot)-u^\Delta((n-1)\Delta,\cdot)=
\sube[u^\Delta((n-1)\Delta,\cdot+\sqrt{\Delta}X+\Delta Y)]-u^\Delta((n-1)\Delta,\cdot).
$$

The main result of the paper is proving the convergence of $u^{\Delta}$ to $u$ and determining its convergence rate by obtaining the
explicit error bounds between the approximate solution and
the viscosity solution of the G-equation.

For this, we impose the following assumptions throughout the paper.

\begin{assumption}\label{assumption1}
(i) The initial data $\phi:\mathbb{R}^d\rightarrow\mathbb{R}$ is bounded from below, and
$\beta$-H\"older continuous for some $\beta\in(0,1]$,
$$|\phi(x)-\phi(y)|\le C_\phi|x-y|^\beta,$$
\text{for} $x, y\in\rd$.

(ii) The random vectors $X$ and $Y$ satisfy the moment conditions: $M_X^3<\infty$ and $M_Y^2<\infty$, with $M_{\xi}^p:=\sube[|\xi|^p]$. Moreover, $X$ has no mean uncertainty, i.e. $\sube[X]=\sube[-X]=0.$

\end{assumption}

We make some comments on the above assumptions.

\begin{remark}
Assumptions (i) and (ii) are standard in the (robust) central limit theorem literature. The regularity of the initial condition $\phi$ implies the regularity of the viscosity solution $u$ (see Lemma \ref{regularity}). The bounded from below property of $\phi$ guarantees the Fatou's property of $\sube$ (see (\ref{Fatou}) or Lemma 2.6 in \cite{CJP}), which will in turn be used to establish an upper bound for the approximation error (see (\ref{super2}) in section \ref{lowbd}).

On the other hand, the moment conditions on $X$ and $Y$ are commonly used in the classical central limit theorem and imply that $M_X^p<\infty$ and $M_Y^q<\infty$ for $0<p<3$ and $0<q<2$. In our setting, they are used to derive the consistency error estimates in section \ref{sec-scheme}.

Finally, we emphasize that there are no independence assumptions made
between $X$ and $Y$. If $X$ and $Y$ are mutually independent, then either $(X,Y)$ must be maximally distributed or one of them is null (see \cite{HU}). The possible dependency between $X$ and $Y$ will be useful when applying the proposed approximation scheme to
the Black-Scholes-Barenblatt equation in section \ref{section: BSB_equ}.

\end{remark}


Under the above assumptions, we prove the following results about the convergence of $u^{\Delta}$ to $u$ and the corresponding convergence rate.

\begin{theorem}\label{maintheorem} Suppose that Assumption \ref{assumption1} is satisfied. Then, the following assertions hold.

(i) (Convergence) The approximate solution $u^{\Delta}\rightarrow u$ {as $\Delta\rightarrow 0$, (locally)} uniformly in $\bar{Q}_T$.

(ii) (Degenerate case) For $\Delta\in(0,1)$, there {exists} a constant $C$ depending only on \textcolor[rgb]{1.00,0.00,0.00}{$d$}, $T$, $C_{\phi}$, $\beta$, $M_X^3$ and $M_Y^2$ such that
$$|u-u^{\Delta}|\le C\Delta^{\beta/6}\ \text{in}\ \bar{Q}_{T}.$$
Furthermore, if the dimension $d=1$ and $T=1$, then the constant $C$ has an explicit formula
\begin{equation}\label{explicit_formula}
C=2124C_{\phi}\left[1+(M_X^3)^{\frac{\beta}{3}}+(M_Y^2)^{\frac{\beta}{2}}\right]
\left[1+(M_X^3)^{\frac23}+M_X^3+(M_Y^2)^{\frac12}+M_Y^2\right].
\end{equation}

(iii) (Non-degenerate case) Furthermore, if the second moment of the random vector $X$ is non-degenerate, i.e. $$\underline{\sigma}^2:=-\sube[-|X|^2]>0,$$
and the initial data $\phi\in\mathcal{C}_b^1(\mathbb{R}^d)$, i.e. $\phi$ is bounded and Lipschitz continuous, then there exists a constant $\alpha\in(0,1)$ {depending on $d$, $\underline{\sigma}^2$ and $M_X^2$} such that $u\in\mathcal{C}_b^{1+\frac{\alpha}{2},2+\alpha}({Q}_{T})$ \textcolor[rgb]{1.00,0.00,0.00}{(see the end of this section for its definition)}. Moreover, for $\Delta\in(0,1)$, {there exists a constant $C$ depending only on $T$, $C_{\phi}$, $\alpha$, $M_X^{2+\alpha}$ and $M_Y^2$ such that}
$$|u-u^{\Delta}|\leq C{\Delta}^{\max\{\frac{\alpha}{2},\frac16\}}\ \ \text{in}\ \bar{Q}_T.$$
\end{theorem}

Assertion (i) is proved in section \ref{sec-convergence}, assertion (ii) is proved in sections \ref{upbd} and \ref{lowbd}, and assertion (iii) is proved in section \ref{sec-special}.
We prove them under the framework of monotone approximation schemes for viscosity solutions.
The first step is to rewrite the recursive approximation (\ref{scheme}) as a monotone scheme (see (\ref{semischeme})), and then derive
the key properties for the monotone scheme in section \ref{sec-scheme}.
It is precisely where the four axioms of the sublinear expectation $\sube$ are
used in an essential way. Using the consistency error estimates derived in section \ref{sebsection: splitting} and the comparison principle for the approximation
scheme established in section \ref{sec-comparison}, we obtain a
lower bound for the approximation error by a mollification procedure.
The upper bound for the approximation
error is further obtained by interchanging the roles of the monotone approximation scheme and the original
G-equation. This depends crucially on the regularity property of the approximate solution established in section \ref{sec-estimate}. Finally, the non-degenerate situation (iii) is proved as a special case of the general (possibly degenerate) situation established in (ii), together with an application of the regularity theory of fully nonlinear PDEs.\\

Next, we give a literature review on monotone schemes followed by the comments on their specification in our setup under G-expectations. Monotone approximation schemes for viscosity solutions were first studied by
Barles and Souganidis \cite{Barles}, who showed that any \emph{monotone,
stable} and \emph{consistent} approximation scheme converges to the correct
solution, provided that there exists a comparison principle for the
limiting equation. The corresponding convergence rate had been an
open problem for a long time until late 1990s when Krylov introduced
the {shaking coefficients technique} to construct a sequence of
smooth subsolutions/supersolutions in \cite{Krylov1} and \cite{Krylov}. This technique was further
developed by Barles and Jacobsen in a sequence of papers (see
\cite{BJ1} and \cite{Jakobsen} and more references therein), and
{has} recently been applied to solve various problems (see, among
others, \cite{Erhan} \cite{BPZ} \cite{FTW}, \cite{HL1} and \cite{Huang}).

Krylov's technique depends crucially on the convexity/concavity of
the underlying equation with respect to its terms. As a result,
unless the approximate solution has enough regularity (so one can
interchange the roles of the approximation scheme and the original
equation), the {shaking coefficients technique} only gives either an
upper or a lower bound for the approximation error, but not both. A
further breakthrough was made by Barles and Jacobsen in \cite{BJ0}
and \cite{BJ}, who combined the ideas of optimal switching
approximation of Hamilton-Jacobi-Bellman equations with the shaking
coefficients technique. They obtained both upper and lower bounds of
the {error estimate}, but with a lower {convergence rate} due to the
introduction of another approximation layer. See also \cite{CS} for its recent development in a bounded domain without any convexity/concavity assumptions.

In the setup of G-expectations, the sublinear expectation $\sube$ and the possible dependency between random vectors $X$ and $Y$ bring in additional difficulties when applying the monotone scheme. Specifically, the regularity properties of (approximate) solutions and the consistency error estimates are derived under the framework of G-expectations (see section \ref{sec-estimate} and Proposition \ref{operator}), where the four axioms of the sublinear expectation $\sube$ are
used in an essential way (without any independency assumption between $X$ and $Y$).
On the other hand, since there are no variable coefficients to shake in order to
apply the mollification procedure to construct the smooth subsolutions/supersolutions, the corresponding convergence rate for the approximate solution to the viscosity solution turns out to be faster than the ones in the PDE literature
(It is ${\beta}/{6}$ in our case).
Moreover, by establishing almost the same regularity property for the approximate solution $u^{\Delta}$ as for the viscosity solution $u$ in Lemma \ref{regularity1}, we are able to interchange the roles of the G-equation and its approximation scheme, and thus obtain a symmetric upper bound and lower bound for the approximation error which is rare in the PDE literature (see \cite{BJ1} and \cite{Jakobsen}). We also work out explicit formulae for all the constants in our estimates. This enables us to derive an \emph{explicit} error bound for the first time, which has a nontrivial application to Peng's robust central limit theorem.\\

We show two applications of Theorem \ref{maintheorem}. The first one is the derivation of the convergence rate for the robust central limit theorem, which is discussed in section \ref{sec-CLT}. Herein, the asymmetric independency of random vectors $\{(X_i, Y_i)\}_{i\geq 1}$ means one cannot simply apply the tower property of expectations as in the classical linear case. We overcome this difficulty by a mathematical induction argument in Theorem \ref{Theorem_CLT}.
Thanks to the explicit formula for the error bound $C$ in (\ref{explicit_formula}), we are able to obtain \emph{an explicit convergence rate of Berry-Esseen type}. To the best of our knowledge, this is the first result about Peng's robust central limit theorem with an explicit bound. The constant $C$ in (\ref{explicit_formula}) is not sharp, and it would be interesting to improve such an explicit bound in the future. The second application is obtaining a numerical approximation scheme for \textcolor[rgb]{1.00,0.00,0.00}{the Black-Scholes-Barenblatt (BSB) equation}, which is widely used to model volatility uncertainty (see \cite{ALP}, \cite{Lyons} and \cite{Vorbrink}). We first make a connection between the G-equation and the BSB equation, and show that the proposed approximation scheme is a natural generalization of the well known Cox-Ross-Rubinstein (CIR) binomial tree approximation to the case with model ambiguity (see section \ref{section: BSB_equ}).

Theorem \ref{maintheorem} may have potential applications to other problems in $G$-expectations. To name a few, it could be applied to derive the convergence rates for generalized robust central limit theorems as considered in \cite{Bayraktar}, \cite{Shiyufeng} and \cite{ChenZengjing}. One of the key steps is to construct appropriate monotone approximation schemes corresponding to the sequence of involved random variables. Another application of Theorem \ref{maintheorem} is to approximate G-expectations as in \cite{Dolinsky} \cite{Dolinsky2} and \cite{Yifan}, which needs the notion of G-Brownian motion developed in \cite{Peng2008(2)}. Finally, the convergence analysis of the monotone approximation schemes may also offer new insight for the numerical solutions of (backward) stochastic differential equations driven by G-Brownian motion (see \cite{Mingshang}, \cite{Mingshang2} and \cite{Yiqing}).\\

The rest of the paper is organized as follows. Sections 2 and 3 are devoted to the applications to the robust central limit theorem and the
Black-Scholes-Barenblatt(BSB) equation, respectively. Section 4 establishes the regularity properties of the viscosity solution and the approximate solution. Sections 5 and 6 are devoted to the monotone approximation scheme and its convergence rate. Section 7 then discusses some special cases.\\


\emph{Notation.} Let $\delta\in(0,1]$. The (semi)norms of
a function $g:\mathbb{R}^d\to\mathbb{R}$ are defined as
$$|g|_{0}:=\sup_{x\in \mathbb{R}^d}|g(x)|, \ \ \ [g]_0 :=\sup_{x\in \mathbb{R}^d}g(x)^- ,\ \ \ [g]_{\mathcal{C}^{\delta}}:=\sup_{\substack{ x,x'\in \mathbb{R}^d \\ x\neq x' }}\frac{|g(x)-g(x')|}{|x-x'|^{\delta}}.$$
Let $\mathcal{C}_{lb}(\mathbb{R}^d)$ be the space of lower bounded continuous functions $g$ on $\mathbb{R}^d$ such that $[g]_0<\infty$,
$\mathcal{C}^{\delta}_{lb}(\mathbb{R}^d)$ be the space of lower bounded continuous functions $g$ on $\mathbb{R}^d$ such that $[g]_0+[g]_{\mathcal{C}^{\delta}}<\infty,$  
and $\mathcal{C}_b^{2+\delta}(\mathbb{R}^d)$ be the space of bounded continuous functions $g$ on $\mathbb{R}^d$ such that $|D_x^ig|_0<\infty$ for $0\leq i\leq 2$ and $[D_x^2g]_{\mathcal{C}^{\delta}}<\infty$.

Similarly, for a function $f:Q_T\to
\mathbb{R}$, we introduce its (semi)norms
$$|f|_{0}:=\sup_{(t,x)\in Q_T}|f(t,x)|, \ \ \ [f]_0:=\sup_{(t,x)\in Q_T}f(t,x)^-,$$
$$[f]_{1,\delta}:=\sup_{\substack{ (t,x),(t',x)\in Q_T \\t\neq t'}}\frac{|f(t,x)-f(t',x)|}{|t-t'|^{\delta}},\ \ \
[f]_{2,\delta}:=\sup_{\substack{ (t,x),(t,x')\in Q_T \\ x\neq x'
}}\frac{|f(t,x)-f(t,x')|}{|x-x'|^{\delta}}.$$ Furthermore,
$[f]_{\mathcal{C}^{{\delta}/{2},\delta}}:=[f]_{1,\delta/2}+[f]_{2,\delta}$.
Let $\mathcal{C}_{lb}(Q_T)$ be the space of lower bounded continuous functions $f$ on $Q_T$ such that $[f]_0<\infty$, $\mathcal{C}_{lb}^{{\delta}/{2},\delta}(Q_T)$ be the space of lower bounded continuous
functions $f$ on $Q_T$ such that $[f]_0+[f]_{\mathcal{C}^{{\delta}/{2},\delta}}<\infty$, 
and $\mathcal{C}_b^{1+\frac{\delta}{2},2+{\delta}}(Q_T)$ be the space of bounded continuous functions $f$ such that $|\partial_t^iD_x^jf|_0<\infty$ for $0\leq i\leq 1$, $0\leq j\leq 2$, and $[\partial_t^{1}f]_{\mathcal{C}^{{\delta}/{2},\delta}}+[  D_x^{2}f]_{\mathcal{C}^{{\delta}/{2},\delta}}<\infty$. Basically, $\mathcal{C}_{lb}^{{\delta}/{2},\delta}(Q_T)$ is the space of viscosity solutions and  $\mathcal{C}_b^{1+\frac{\delta}{2},2+{\delta}}(Q_T)$ is the space of classical solutions in this paper.

Finally, for $S=\mathbb{R}^d$ or ${Q}_T$, we
denote by $\mathcal{C}_{lb}^{\infty}(S)$ be the spaces of lower bounded continuous functions on $S$
with bounded derivatives of any order.


\section{Application to robust central limit theorem}\label{sec-CLT}

In this section, we apply Theorem \ref{maintheorem} to derive the convergence rate (with an explicit bound of Berry-Esseen type) of the celebrated robust central limit theorem introduced in \cite{Peng2008}. For this, let $\{(X_i,Y_i)\}_{i\geq 1}$ be a sequence of $\rd\times\rd$-valued random vectors defined on $(\Omega,\mathcal{H},\sube)$ such that $(X_1,Y_1)=(X,Y)$, $(X_{i+1},Y_{i+1}) \stackrel{d}{=} (X_i,Y_i)$ and $(X_{i+1},Y_{i+1})$ is independent of $\{(X_1,Y_1),...,(X_i,Y_i)\}$ for each $i\in\mathbb{N}$. Furthermore, assume that $X$ and $Y$ satisfy Assumption \ref{assumption1}(ii). Then, Peng proved that the sequence $\{S_n\}_{n\geq 1}$ defined by
\begin{equation}\label{s_n_def}
S_n:=\sum_{i=1}^n(\frac{X_i}{\sqrt{n}}+\frac{Y_i}{n})
\end{equation}
converges in law to $(\xi+\zeta)$:
\begin{equation}\label{clt_peng}
\lim_{n\rightarrow\infty}\sube[\phi(S_n)]=\widetilde{\mathbb{E}}[\phi(\xi+\zeta)].
\end{equation}
for any continuous test function satisfying linear growth condition, \textcolor[rgb]{1.00,0.00,0.00}{where $(\xi,\zeta)$ follows G-distribution under another sublinear expectation $\widetilde{\mathbb{E}}$ possibly different from $\sube$.} See Theorem 5.1 in \cite{Peng2008} for its proof.

Following Peng's seminal work, a lot of efforts have been made to further obtain the various convergence rates of (\ref{clt_peng}) with additional model assumptions (see, for example, \cite{FPSS} \cite{Krylov4} and \cite{Song}). However, the existing literature on the convergence rates of (\ref{clt_peng}) assumes that either $X_i=0$ or $Y_i=0$ and, to be best of our knowledge, the convergence rate of (\ref{clt_peng}) for the general situation (i.e. $X_i\neq 0$ and $Y_i\neq 0$) and an explicit bound of Berry-Esseen type
are still lacking. Our aim is therefore to obtain a general result about the convergence rate of (\ref{clt_peng}) with an explicit bound using Theorems \ref{maintheorem}.

To illustrate how it works, we provide some
preliminary informal arguments to highlight the main ideas and build intuition. Consider $d=1$ for simplicity. If we replace the sublinear expectation $\sube$ with the linear expectation $\mathbb{E}$ and let $\{(X_i,Y_i)\}_{i\geq 1}$ be {a sequence of i.i.d.} copies of $(X,Y)$ such that
$\mathbb{E}[X]=0$, then the recursive approximation (\ref{scheme}) reduces to
\begin{equation*}
u^\Delta(n\Delta,x)=\mathbb{E}[\phi(x+\sum_{i=1}^n (\sqrt{\Delta}X_i+\Delta Y_i))].
\end{equation*}
On the other hand, the nonlinear function $G$ defined in (\ref{functionG}) reduces to $G(p,A)=\frac12 \mathbb{E}[|X|^2]A+\mathbb{E}[Y]p$, so $\mathbb{E}[|X|^2]$ and $\mathbb{E}[Y]$ turn out to be the coefficients of the linear equation
$$\partial_tu-\frac12\mathbb{E}[|X|^2]\partial_{xx}u-\mathbb{E}[Y]\partial_xu=0.$$
The Feynman-Kac formula then implies that
\begin{equation*}
u(t,x)={\mathbb{E}}[\phi(x+\sqrt{t}\xi+t\zeta)],
\end{equation*}
where $\xi\sim N(0, \mathbb{E}[|X|^2])$ and $\mathbb{\zeta}=\mathbb{E}[Y]$. Taking $\Delta=\frac1n$ and using Theorem \ref{maintheorem}, we obtain $$u^{\Delta}(1,0)=\mathbb{E}\left[\phi\left(\sum_{i=1}^n (\frac{X_i}{\sqrt{n}}+\frac{Y_i}{n})\right)\right]\rightarrow u(1,0)={\mathbb{E}}[\phi(\xi+\zeta)],$$
which is precisely the classical central limit theorem (for $\xi$) and law of large numbers (for $\zeta$)

\begin{theorem}\label{Theorem_CLT}
Let
$\{S_n\}_{n\geq 1}$ be given as in (\ref{s_n_def}), and suppose  Assumption \ref{assumption1} is satisfied.
Then, the following assertions hold.

(i) (Degenerate case) There exists a constant $C$ depending only on $T$, $C_{\phi}$, $\beta$, $M_X^{3}$ and $M_Y^2$ such that
\begin{equation}\label{convergence_rate}
\left|\sube[\phi(S_n)]-\widetilde{\mathbb{E}}[\phi(\xi+\zeta)]\right|\leq Cn^{-\frac{\beta}{6}}.
\end{equation}
Moreover, if the dimension $d=1$ and $T=1$, then the constant $C$ has an explicit formula given in (\ref{explicit_formula}).

(ii) (Non-degenerate case) Moreover, if the second moment of the random vector $X$ is non-degenerate, i.e. $$\underline{\sigma}^2:=-\sube[-|X|^2]>0,$$
and the initial data $\phi\in\mathcal{C}_b^1(\mathbb{R}^d)$, i.e. $\phi$ is bounded and Lipschitz continuous, then there exists a constant $\alpha\in(0,1)$ {depending on $d$, $\underline{\sigma}^2$ and $M_X^2$} and a constant $C$  {depending only on $T$, $C_{\phi}$, $\alpha$, $M_X^{2+\alpha}$ and $M_Y^2$ such that}
 \begin{equation}\label{convergence_rate_2}
\left|\sube[\phi(S_n)]-\widetilde{\mathbb{E}}[\phi(\xi+\zeta)]\right|\leq Cn^{-\max\{\frac{\alpha}{2},\frac{1}{6}\}}.
\end{equation}
 \end{theorem}

\begin{proof} We claim that, for all $n\in\mathbb{N}$ such that $n\Delta\le T$ and $x\in\rd$,
\begin{equation}\label{clt_representation}
u^\Delta(n\Delta,x)=\sube[\phi(x+\sum_{i=1}^n (\sqrt{\Delta}X_i+\Delta Y_i))].
\end{equation}
If the representation formula (\ref{clt_representation}) holds, then by letting $\Delta=1/n$ and $x=0$, we obtain
$$u^{\Delta}(1,0)=\sube[\phi(S_n)].$$
On the other hand, the representation formula (\ref{solution}) implies that
\begin{equation*}
u(1,0)=\widetilde{\mathbb{E}}[\phi(\xi+\zeta)].
\end{equation*}
Hence, the assertions (i) and (ii) follow from Theorem \ref{maintheorem}.

We are left to show (\ref{clt_representation}). We prove by induction on $n$. Note that the case $n=1$ follows directly from (\ref{scheme}). Next, we claim that for all $n\in\mathbb{N}$ and $g\in\mathcal{C}_{lb}(\rd)$,
\begin{equation}\label{claim11}
\sube\left[g\left(\sum_{i=1}^n (\sqrt{\Delta}X_i+\Delta Y_i)\right)\right]=\sube\left[g\left(\sum_{i=2}^{n+1} (\sqrt{\Delta}X_i+\Delta Y_i)\right)\right],
\end{equation}
and suppose (\ref{clt_representation}) holds for $n\in\mathbb{N}$ such that $n\Delta\le T$. Then, if $(n+1)\Delta\le T$, we use (\ref{claim11}) to obtain
\begin{align*}
&\ u^\Delta((n+1)\Delta,x)\\
= &\
\sube[u^\Delta(n\Delta,x+\sqrt{\Delta}X+\Delta Y)] \\
= &\
\sube\left[\sube\left[\phi(x+\sqrt{\Delta}p+\Delta q+\sum_{i=1}^n (\sqrt{\Delta}X_i+\Delta Y_i))\right]_{(p,q)=(X,Y)}\right] \\
= &\
\sube\left[\sube\left[\phi(x+\sqrt{\Delta}p+\Delta q+\sum_{i=2}^{n+1} (\sqrt{\Delta}X_i+\Delta Y_i))\right]_{(p,q)=(X,Y)}\right] \\
= &\
\sube[\phi(x+\sum_{i=1}^{n+1} (\sqrt{\Delta}X_i+\Delta Y_i))].
\end{align*}
In other words, (\ref{clt_representation}) also holds for $n+1$.

Finally, to show (\ref{claim11}), we prove again by induction on $n$. The case $n=1$ follows from $(X_{2},Y_{2}) \stackrel{d}{=} (X_1,Y_1)$. Suppose  (\ref{claim11}) holds for $n\in\mathbb{N}$, then
\begin{align*}
&\sube\left[g\left(\sum_{i=1}^{n+1} (\sqrt{\Delta}X_i+\Delta Y_i)\right)\right]
=
\sube\left[g\left(\sum_{i=1}^n (\sqrt{\Delta}X_i+\Delta Y_i)+\sqrt{\Delta}X_{n+1}+\Delta Y_{n+1}\right)\right].
\end{align*}
Since $(X_{n+1},Y_{n+1})$ is independent of $\{(X_1,Y_1),...,(X_n,Y_n)\}$, The RHS of the above equality further equals to
\begin{align*}
&\
\sube\left[\sube\left[g\left(\sum_{i=1}^n (\sqrt{\Delta}x_i+\Delta y_i)+\sqrt{\Delta}X_{n+1}+\Delta Y_{n+1}\right)\right]_{(x_i,y_i)=(X_i,Y_i),i=1,...,n}\right]\\
= &\
\sube\left[\sube\left[g\left(\sum_{i=1}^n (\sqrt{\Delta}x_i+\Delta y_i)+\sqrt{\Delta}X_{n+2}+\Delta Y_{n+2}\right)\right]_{(x_i,y_i)=(X_i,Y_i),i=1,...,n}\right]\\
= &\
\sube\left[f\left(\sum_{i=1}^n (\sqrt{\Delta}X_i+\Delta Y_i)\right)\right],
\end{align*}
where \textcolor[rgb]{1.00,0.00,0.00}{$f(x):=\sube\left[g\left(x+\sqrt{\Delta}X_{n+2}+\Delta Y_{n+2}\right)\right]$} and the first equality follows from $(X_{n+2},Y_{n+2}) \stackrel{d}{=} (X_{n+1},Y_{n+1})$. In turn, since (\ref{claim11}) holds for $n$, we further have
\begin{align*}
&\sube\left[g\left(\sum_{i=1}^{n+1} (\sqrt{\Delta}X_i+\Delta Y_i)\right)\right]\\
= &\
\sube\left[f\left(\sum_{i=2}^{n+1} (\sqrt{\Delta}X_i+\Delta Y_i)\right)\right]\\
= &\
\sube\left[\sube\left[g\left(\sum_{i=2}^{n+1} (\sqrt{\Delta}x_i+\Delta y_i)+\sqrt{\Delta}X_{n+2}+\Delta Y_{n+2}\right)\right]_{(x_i,y_i)=(X_i,Y_i),i=2,...,n+1}\right]\\
= &\
\sube\left[g\left(\sum_{i=2}^{n+2} (\sqrt{\Delta}X_i+\Delta Y_i)\right)\right],
\end{align*}
which completes the proof.
\end{proof}

\section{Application to Black-Scholes-Barenblatt equation}\label{section: BSB_equ}
In this section, we apply the approximation scheme (\ref{scheme}) to the  Black-Scholes-Barenblatt equation (see \cite{ALP}, \cite{Lyons} and \cite{Vorbrink} for the dimension $d=1$), which often arises from option pricing models with volatility uncertainty, namely
\begin{equation}\label{BSB}
\partial_tu+rx\partial_xu+\frac12 \bar{\sigma}^2x^2\partial_{xx}u^+-\frac12 \underline{\sigma}^2x^2\partial_{xx}u^--ru=0, \ \ \ u|_{t=T}=\Phi,
\end{equation}
where $r$ is the constant riskless interest rate, $\bar{\sigma}\ge\underline{\sigma}>0$ are two constants representing upper and lower bounds on the volatility of underlying price, and $\Phi:\mathbb{R}\to\mathbb{R}$ represents some European contingent claim payoff function. Note that when $\bar{\sigma}=\underline{\sigma}$, the equation reduces to the classical Black-Scholes equation.

To apply the approximation scheme (\ref{scheme}), some transformations are needed firstly: let $v(t,x):=u(T-t,e^x)e^{rt}$, then (\ref{BSB}) becomes
\begin{equation}\label{BSB1}
\partial_tv-\sup_{\sigma\in[\underline{\sigma},\bar{\sigma}]}\left\{(r-\frac12\sigma^2)\partial_xv+\frac12 \sigma^2\partial_{xx}v\right\}=0, \ \ \ v|_{t=0}=\Phi(e^x).
\end{equation}
Comparing the equation (\ref{BSB1}) to the G-equation (\ref{functionG}) and (\ref{PDE_1}), we only need to construct a sublinear expectation $\sube$ and find random variables $(X,Y)$ with $X$ having no mean uncertainty such that
\begin{equation}\label{Gexample}
G(p,A)=\sube[pY+\frac12AX^2]=\sup_{\sigma\in[\underline{\sigma},\bar{\sigma}]}\left\{(r-\frac12\sigma^2)p+\frac12 \sigma^2A\right\}.
\end{equation}
To this end, {suppose we are given a measurable space $(\Omega,\mathcal{F})$ which supports a class of probability measures $\mathbb{P}^\sigma$ for $\sigma\in[\underline{\sigma},\bar{\sigma}]$}. We can then define a random variable $X$ such that $\mathbb{P}^\sigma(X=\sigma)=\frac12$ and $\mathbb{P}^\sigma(X=-\sigma)=\frac12$ for any $\sigma\in[\underline{\sigma},\bar{\sigma}]$, and a random variable $Y:=r-\frac12X^2$. Consequently, a sublinear expectation space $(\Omega, \mathcal{H},\sube)$ can be defined such that $X,Y\in\mathcal{H}$ and
$$\sube\left[\xi\right]=\sup_{\sigma\in[\underline{\sigma},\bar{\sigma}]}\mathbb{E}^{\mathbb{P}^\sigma}[\xi] \ \ \text{for} \ \ \xi\in\mathcal{H}.$$
 It is clear that $\sube[X]=\sube[-X]=0$ and (\ref{Gexample}) holds. The approximation scheme (\ref{scheme}) then has a simple form:
\begin{align*}
v^\Delta(t,x) = &\ \sube[v^\Delta(t-\Delta,x+\sqrt{\Delta}X+\Delta Y)] \\
= &\
\sup_{\sigma\in[\underline{\sigma},\bar{\sigma}]}\mathbb{E}^{\mathbb{P}^\sigma}[v^\Delta(t-\Delta,x+\sqrt{\Delta}X+\Delta(r-\frac12X^2))] \\
= &\
\sup_{\sigma\in[\underline{\sigma},\bar{\sigma}]}\left[\frac12v^\Delta\left(t-\Delta,x+(r-\frac12\sigma^2)\Delta+\sigma\sqrt{\Delta}\right)+\frac12v^\Delta\left(t-\Delta,x+(r-\frac12\sigma^2)\Delta-\sigma\sqrt{\Delta}\right)\right]
\end{align*}
for $\Delta\le t\le T$ and $v^\Delta(t,x)=\Phi(e^x)$ for $t<\Delta$.

\begin{remark}
When $\bar{\sigma}=\underline{\sigma}=\sigma>0$, the above approximation scheme  reduces to classical Cox, Ross and Rubinstein (CRR) binomial tree approximation for $X_t=\ln(S_t)$, with $S_t$ following the geometric Brownian motion $dS_t=rS_tdt+\sigma S_tdW_t$. Indeed, by It\^{o}'s lemma, we have $dX_t=(r-\frac12\sigma^2)dt+\sigma dW_t$. The CRR binomial tree approximation for $X_t$ is then given as follows (we only display one step binomial tree approximation for simplicity):

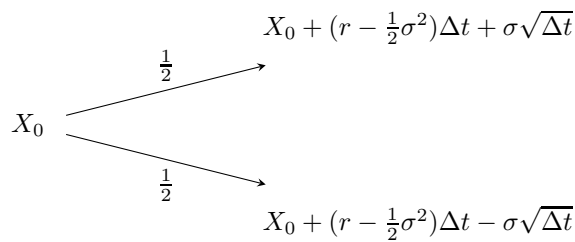
\begin{center}
\begin{tikzpicture}[>=stealth]
    \matrix (tree) [
      matrix of nodes,
      minimum size=1cm,
      column sep=2.5cm,
      row sep=0.3cm,
    ]
    {
          & $X_0+(r-\frac12\sigma^2)\Delta t+\sigma\sqrt{\Delta t}$    \\
      $X_0$ &    \\
          & $X_0+(r-\frac12\sigma^2)\Delta t-\sigma\sqrt{\Delta t}$    \\
              };
    \draw[->] (tree-2-1) -- (tree-1-2) node [midway,above] {$\frac12$};
    \draw[->] (tree-2-1) -- (tree-3-2) node [midway,below] {$\frac12$};
\end{tikzpicture}
\end{center}
The approximation scheme for the approximation of $v(t,x)$ then reduces to the CRR binomial tree approximation
$$v^{\Delta}(t,x)=\frac12v^\Delta\left(t-\Delta,x+(r-\frac12\sigma^2)\Delta+\sigma\sqrt{\Delta}\right)+\frac12v^\Delta\left(t-\Delta,x+(r-\frac12\sigma^2)\Delta-\sigma\sqrt{\Delta}\right).$$
\end{remark}

Since Assumption \ref{assumption1}(ii) clearly holds, if the composition function $\phi(x):=\Phi(e^x)$ satisfies Assumption \ref{assumption1}(i), Theorem \ref{maintheorem} then implies that $v^{\Delta}$ converges to $v$ locally uniformly. Notice that with our construction of the sublinear expectation space $(\Omega,\mathcal{H},\hat{\mathbb{E}})$ and our choice of the random variable $X$, it also holds that
\begin{equation}\label{extra_assumption}
\hat{\mathbb{E}}[X]=\hat{\mathbb{E}}[-X]=\hat{\mathbb{E}}[X^{3}]=\hat{\mathbb{E}}[-X^{3}]=0,\ \ \ M_X^{4}=\hat{\mathbb{E}}[|X|^{4}]<\infty.
\end{equation}
Thanks to the above properties, we obtain a better convergence rate than {that} in Theorem \ref{maintheorem}{(iii)}.

\begin{proposition}\label{BSB_theorem} Let $(X,Y)$ be the random variables constructed as above on the sublinear expectation space $(\Omega,\mathcal{H},\sube)$. Then, for any test function $\phi(x):=\Phi(e^{x})\in\mathcal{C}_b^1(\mathbb{R})$, there exists a constant $\alpha\in(0,1)$ {depending on $\underline{\sigma}^2$ and $M_X^2$} such that $v\in\mathcal{C}_b^{1+\frac{\alpha}{2},2+\alpha}({Q}_{T})$. Moreover, for $\Delta\in(0,1)$, {there exists a constant $C$ depending only on $T$, $C_{\phi}$, $\alpha$, and $M_X^{2+\alpha}$ such that}
\begin{equation}\label{BSB_convergence}
|v-v^{\Delta}|\leq C\Delta^{\max\{\frac{\alpha}{2},\frac14\}}\ \ \ in\ \bar{Q}_T.
\end{equation}
\end{proposition}

The proof follows along a similar argument and procedure used in Theorem \ref{maintheorem} with a refinement of the consistency error estimates in Proposition \ref{operator}, and is therefore postponed to the Appendix.

\begin{remark}
In the general degenerate situation (i.e. without the assumption that $\underline{\sigma}>0$), the convergence rate in (\ref{BSB_convergence}) becomes $\frac14$.
Note that (\ref{extra_assumption}) is the only condition on $X$ needed to obtain the convergence rate $\frac14$.
The condition (\ref{extra_assumption}) is also imposed in \cite{Krylov4} and the same convergence rate is obtained in Theorem 4.1 therein. However, there is no $Y$ component {in} \cite{Krylov4}, so {the result therein} is only a special case of our situation.
\end{remark}

\section{Regularity estimates} \label{sec-estimate}

We establish the space and time regularity properties of both $u$ and $u^{\Delta}$, which are crucial for proving the convergence of $u^{\Delta}$ to $u$ and determining its convergence rate.
In particular, the regularity of $u$ and $u^{\Delta}$ will play a vital role in mollification procedures
(see (\ref{lowerbound_2}) in section \ref{upbd} and (\ref{upperbound_2}) in section \ref{lowbd}).

\begin{lemma}\label{regularity0}
Suppose that Assumption \ref{assumption1}(ii) is satisfied. Then,

(i) $\widetilde{\mathbb{E}}[|\xi|^2]=\sube[|X|^2]=M_X^2$,

(ii) $\widetilde{\mathbb{E}}[|\zeta|^p]\le\sube[|Y|]^p=(M_Y^1)^p$ for $p>0$,

\noindent
where $(\xi,\zeta)$ is a pair of G-distributed random vectors characterized by (\ref{functionG}) with the associated sublinear expectation $\widetilde{\mathbb{E}}$.
\end{lemma}

\begin{proof}
Assertion (i) is obvious by combining (\ref{functionG}) and (\ref{Gdistributed}), and letting $p=0$, $A=2I_d$. By (\ref{maximal}), we further have
$$\sube[\langle p,Y\rangle]=G(p,0)=\widetilde{\mathbb{E}}[\langle p,\zeta\rangle]=\max_{q\in Q}\langle p,q\rangle$$
for any $p\in\mathbb{R}^d$. Then, for any $q\in Q$,
\begin{align*}
|q|^2 \le &\ \max_{q'\in Q}\langle q, q' \rangle=\sube[\langle q, Y\rangle]\le\sube[\max_{q\in Q}\langle q, Y\rangle] \\
= &\
\sube\left[\sube[\langle y, Y\rangle]|_{y=Y}\right]\le\sube\left[\sube[|y||Y|]\left|_{y=Y}\right.\right]=\sube\left[|Y|\sube[|Y|]\right]=\sube\left[|Y|\right]^2.
\end{align*}
Thus, we obtain from (\ref{maximal}) again that $\widetilde{\mathbb{E}}[|\zeta|^p]=\max_{q\in Q}|q|^p\le\sube[|Y|]^p$.
\end{proof}

\begin{lemma}\label{regularity}
Suppose that Assumption \ref{assumption1} is satisfied. Then, for any $t,s\in[0,T]$ and $x,y\in\rd$,

(i) $|u(t,x)-u(t,y)|\le C_\phi |x-y|^\beta$.

(ii) $|u(s,x)-u(t,x)|\le C_\phi K_0|s-t|^{\beta/2}$,
where the constant $K_0$ is defined as
\begin{equation}\label{constant_K_0}
K_0:=e^{\frac{\beta T}{2}}[(M_X^2)^{\frac{\beta}{2}}+(M_Y^2)^{\frac{\beta}{2}}].
\end{equation}
\end{lemma}

\begin{proof}
Assertion (i) is a direct consequence of the representation formula (\ref{solution}), the sub-additivity of $\widetilde{\mathbb{E}}$ and the H\"older continuity of $\phi$.

To prove (ii), we may assume $t\le s$. Note that the semigroup property of $u$ implies that
\begin{equation}\label{dpp}
u(s,x)=\widetilde{\mathbb{E}}[u(t,x+\sqrt{s-t}\xi+(s-t)\zeta)].
\end{equation}
In turn, the sub-additivity of $\widetilde{\mathbb{E}}$ and (i) yield
\begin{align*}
 |u(s,x)-u(t,x)| \le
 &\
 \widetilde{\mathbb{E}}[|u(t,x+\sqrt{s-t}\xi+(s-t)\zeta)-u(t,x)|] \\
 \le &\
 \widetilde{\mathbb{E}}[|u(t,x+\sqrt{s-t}\xi+(s-t)\zeta)-u(t,x+\sqrt{s-t}\xi)|] \\
 &\ +
 \widetilde{\mathbb{E}}[|u(t,x+\sqrt{s-t}\xi)-u(t,x)|] \\
 \le &\
 \widetilde{\mathbb{E}}[C_\phi |(s-t)\zeta|^\beta]+\widetilde{\mathbb{E}}[C_\phi |\sqrt{s-t}\xi|^\beta] \\
 = &\ C_\phi(\widetilde{\mathbb{E}}[|\xi|^\beta]+|s-t|^{\beta/2}\widetilde{\mathbb{E}}[|\zeta|^\beta])|s-t|^{\beta/2} \\
 \le &\
 C_\phi((M_X^2)^{\beta/2}+|s-t|^{\beta/2}(M_Y^1)^\beta)|s-t|^{\beta/2}.
\end{align*}
where we used Lemma \ref{regularity0} and the fact that $\widetilde{\mathbb{E}}[|\xi|^\beta]\le \widetilde{\mathbb{E}}[|\xi|^2]^{\beta/2}$ in the last inequality. The conclusion then follows from the inequalities
$$(M_X^2)^{\beta/2}+|s-t|^{\beta/2}(M_Y^1)^\beta\leq (M_X^2)^{\beta/2}+T^{\beta/2}(M_Y^2)^{{\beta}/{2}}\leq K_0.$$
\end{proof}

\begin{lemma}\label{regularity1}
Suppose that Assumption \ref{assumption1} is satisfied. Then, for any $\Delta\in(0,1)$, $t,s\in[0,T]$ and $x,y\in\rd$,

(i) $|u^\Delta(t,x)-u^\Delta(t,y)|\le C_\phi |x-y|^\beta$.

(ii) $|u^\Delta(s,x)-u^\Delta(t,x)|\le \sqrt{3}C_\phi K_0(|s-t|^{\beta/2}+\Delta^{\beta/2})$,
where the constant $K_0$ is given in (\ref{constant_K_0}).
\end{lemma}
\begin{proof}
We first establish the estimate (i) using induction. It is clear that the estimate holds for $t\in[0,\Delta)$. In general, suppose the estimate holds for $t\in[(n-1)\Delta,n\Delta)$ with $n\Delta\leq T$. Then, for $t\in[n\Delta,((n+1)\Delta)\wedge T)$, the sub-additivity of $\sube$ yields
\begin{align*}
  |u^{\Delta}(t,x)-u^{\Delta}(t,y)|&=\left|\sube
  [u^{\Delta}(t-\Delta,x+\sqrt{\Delta}X+\Delta Y)]-\sube
  [u^{\Delta}(t-\Delta,y+\sqrt{\Delta}X+\Delta Y)]\right|\\
  &\leq \sube\left|u^{\Delta}(t-\Delta,x+\sqrt{\Delta}X+\Delta Y)-u^{\Delta}(t-\Delta,y+\sqrt{\Delta}X+\Delta Y)\right|\\
  &\leq \sube[C_{\phi}|x-y|^{\beta}]=C_{\phi}|x-y|^{\beta}.
\end{align*}
where we also used the constant preserving property in the last inequality.

To establish the time regularity for $u^{\Delta}$ in (ii), we divide its proof into four steps.

\emph{Step 1.} We lift the H\"older exponent $\beta$ to $2$ in the estimate (i). Note that the Young's inequality implies that
$$xy\leq \frac{\beta}{2}x^{\frac{2}{\beta}}+\frac{2-\beta}{2}y^{\frac{2}{2-\beta}},\ x,y\geq 0.$$
In turn, for $\alpha\geq 0$ and $\varepsilon>0$, let $x=\alpha^{\beta}$ and $y=\frac{1}{\varepsilon}$, and we have
$$\alpha^{\beta}\leq \frac{\beta}{2}\varepsilon\alpha^2+\frac{2-\beta}{2}\varepsilon^{\frac{-\beta}{2-\beta}},$$
Hence, it follows from (i) that
\begin{equation}\label{estimate1}
u^{\Delta}(t,x)\leq u^{\Delta}(t,y)+a|x-y|^{2}+b,\ x,y\in\mathbb{R}^d,
\end{equation}
where $a:=C_{\phi}\frac{\beta}{2}\varepsilon$ and $b:=C_{\phi}\frac{2-\beta}{2}\varepsilon^{\frac{-\beta}{2-\beta}}$.

\emph{Step 2.} Define $T_\Delta:=\{k\Delta:k\in\mathbb{N}\}$. Then, for $\tau\in[0,T)\cap T_\Delta$ and $k\in\mathbb{N}$ such that $\tau+k\Delta\le T$, we aim to show that
\begin{equation}\label{claim}
u^\Delta(\tau+k\Delta,x)\le u^\Delta(\tau,y)+a(1+\Delta)^k|x-y|^2+aNe^Tk\Delta+b,
\end{equation}
with $a$ and $b$ given in (\ref{estimate1}) and $N:=2M_X^2+3M_Y^2$. Indeed, it is clear that (\ref{claim}) holds for $k=0$. Suppose (\ref{claim}) holds for some $k\in\mathbb{N}$, then,
\begin{align}\label{estimate_11}
u^\Delta(\tau+(k+1)\Delta,x)= &\ \sube[u^\Delta(\tau+k\Delta,x+\sqrt{\Delta}X+\Delta Y)] \notag\\
 \le &\
 u^\Delta(\tau,y)+a(1+\Delta)^k\sube[|x-y+\sqrt{\Delta}X+\Delta Y|^2]\notag\\
 &\ +
 aNe^Tk\Delta+b.
\end{align}
For the sublinear expectation on the RHS of (\ref{estimate_11}), we have
\begin{align*}
&\sube[|x-y+\sqrt{\Delta}X+\Delta Y|^2]\\
    \leq&\ |x-y|^2+2\sube[\langle x-y,\sqrt{\Delta}X\rangle]+2\Delta\sube[\langle x-y,Y\rangle]+\Delta M_X^2+\Delta^2M_Y^2+2\Delta^{\frac32}\sube[\langle X,Y\rangle].
\end{align*}
Since $X$ has no mean uncertainty (cf. Assumption \ref{assumption1}(ii)), it follows that  $\sube[\langle x-y,\sqrt{\Delta}X\rangle]=0$. Furthermore, since
$2\langle x-y,Y\rangle\leq|x-y|^2+|Y|^2$ and $2\langle X,Y\rangle\leq|X|^2+|Y|^2,$
\begin{equation}\label{estimate_12}
\sube[|x-y+\sqrt{\Delta}X+\Delta Y|^2]
\leq(1+\Delta)|x-y|^2+\Delta (2M_X^2+3M_Y^2)=(1+\Delta)|x-y|^2+\Delta N.
\end{equation}
Combining (\ref{estimate_11})-(\ref{estimate_12}) and the fact that $(1+\Delta)^k\le(1+\Delta)^{T/\Delta}\le e^T$, we have
\begin{align*}
u^\Delta(\tau+(k+1)\Delta,x)
 \le &\
 u^\Delta(\tau,y)+a(1+\Delta)^{k+1}|x-y|^2 +
 aNe^T(k+1)\Delta+b,
\end{align*}
which shows that (\ref{claim}) also holds for $(k+1)$.

\emph{Step 3.} We show that the estimate (ii) holds on $t,s\in [0,T)\cap T_{\Delta}$. Indeed, taking $y=x$ in (\ref{claim}), we obtain $$u^\Delta(\tau+k\Delta,x)\le u^\Delta(\tau,x)+C_{\phi}\frac{\beta}{2}\varepsilon Ne^Tk\Delta+C_{\phi}\frac{2-\beta}{2}\varepsilon^{\frac{-\beta}{2-\beta}},$$
for any $\varepsilon>0$. Minimizing the RHS of the above inequality over $\varepsilon$ then yields
$$u^\Delta(\tau+k\Delta,x)\le u^\Delta(\tau,x)+C_{\phi}(Ne^T)^{\frac{\beta}{2}}(k\Delta)^{\frac{\beta}{2}}.$$

\emph{Step 4.} In general, for $s,t\in[0,T]$ \textcolor[rgb]{1.00,0.00,0.00}{such that $s\geq t$}, let $\delta_s$, $\delta_t\in[0,\Delta)$ such that $s-\delta_s$, $t-\delta_t\in T_\Delta$. Then, from (\ref{interp}) we have
\begin{align*}
u^\Delta(s,x)=u^\Delta(s-\delta_s,x) \le &\
 u^\Delta(t-\delta_t,x)+C_{\phi}(Ne^T)^{\frac{\beta}{2}}(s-t-\delta_s+\delta_t)^{\beta/2}\\
 \le &\
 u^\Delta(t,x)+C_{\phi}(Ne^T)^{\frac{\beta}{2}}((s-t)^{\beta/2}+\Delta^{\beta/2}).
\end{align*}
Similarly, we also have, \textcolor[rgb]{1.00,0.00,0.00}{for $t\geq s$ and $s,t\in[0,T]$,}
$$u^\Delta(t,x) \le u^\Delta(s,x)+C_{\phi}(Ne^T)^{\frac{\beta}{2}}((s-t)^{\beta/2}+\Delta^{\beta/2}),$$
from which we then conclude by observing that $(Ne^T)^{\frac{\beta}{2}}\leq \sqrt{3}K_0$.
\end{proof}

\begin{remark}
Note that $u^{\Delta}$ is a piecewise constant approximation of $u$, so it is not continuous in time (with jumps at the partition points $\tau\in T_{\Delta}$). The discontinuity leads to the additional term $\Delta^{{\beta}/{2}}$ in the time regularity of $u^{\Delta}$. Such type of time regularity property also appears in Lemma 2.2 of \cite{Krylov4} in a stochastic control setting. Our regularity result could be regarded as a generalization of \cite{Krylov4} to the sublinear expectation setting.
\end{remark}

\section{A monotone approximation scheme for the G-equation}\label{sec-scheme}

The proof of Theorem \ref{maintheorem} is based on the monotone schemes for viscosity solutions, the framework of which was first introduced by Barles and Souganidis \cite{Barles}. Hence, we first rewrite the recursive approximation (\ref{scheme}) as a monotone scheme, and then derive its consistency error estimates.


Recall that $\mathcal{C}_{lb}(\rd)$ is the space of lower bounded continuous functions on $\rd$. We define a forward operator on $\mathcal{C}_{lb}(\rd)$ as
$$\mathbf{S}(\Delta)\psi(x)=\sube[\psi(x+\sqrt{\Delta}X+\Delta Y)],\ \psi\in\mathcal{C}_{lb}(\mathbb{R}^{d}).$$
Then, from the properties of the sublinear expectation $\sube$, we immediately deduce that the forward operator $\mathbf{S}(\Delta)$ satisfies

(i) \emph{(Monotonicity)} For any
$\psi'\in\mathcal{C}_{lb}(\mathbb{R}^{d})$ with $\psi'\geq \psi$,
\begin{equation}\label{axiom1}
\mathbf{S}(\Delta)\psi'\geq\mathbf{S}(\Delta)\psi.
\end{equation}

(ii) \emph{(Constant preserving)} For any $c\in\mathbb{R}$,
\begin{equation}\label{axiom2}
\mathbf{S}(\Delta)(\psi+c)= \mathbf{S}(\Delta)\psi+c.
\end{equation}

(iii) \emph{(Sub-additivity)} For any $\psi'\in\mathcal{C}_{lb}(\mathbb{R}^{d})$,
\begin{equation}\label{axiom3}
\mathbf{S}(\Delta)(\psi'+\psi)\leq\mathbf{S}(\Delta)\psi'+\mathbf{S}(\Delta)\psi.
\end{equation}

(iv) \emph{(Positive homogeneity)} For any $\lambda\geq 0$,
\begin{equation}\label{axiom4}
\mathbf{S}(\Delta)(\lambda\psi)=\lambda\mathbf{S}(\Delta)\psi.
\end{equation}
Note that (iii) and (iv) imply that $\mathbf{S}(\Delta)\psi$ is convex in $\psi$. On the other hand, the lower boundedness of $\psi$ guarantee the Fatou's property (see Lemma 2.6 in \cite{CJP}): Let $\psi_n\in\mathcal{C}_{lb}(\mathbb{R}^d)$ converges uniformly to $\psi$, then
\begin{equation}\label{Fatou}
\mathbf{S}(\Delta)\psi(x)\leq \liminf_{n}\mathbf{S}(\Delta)\psi_n(x).
\end{equation}

The following error estimates play a vital rule to derive the consistency error estimates for the monotone approximation scheme introduced in section \ref{sebsection: splitting} (see Proposition \ref{scheme property}(iii)).

\begin{proposition}\label{operator} Suppose that Assumption \ref{assumption1}(ii) is satisfied. For $\Delta\in(0,1)$, define
\begin{equation}\label{operator_E}
\mathcal{E}(\Delta,\psi):=
\left|\frac{\mathbf{S}(\Delta)\psi-\psi}{\Delta}-G(D\psi,D^2\psi)\right|_{0}.
\end{equation}

(i) If $\psi\in\mathcal{C}^{2+\alpha}_{b}(\rd)$ for some $\alpha\in(0,1)$, then
$$\mathcal{E}(\Delta,\psi)\leq \ \Delta^{\alpha/2}[D^2\psi]_{\mathcal{C}^\alpha}M_X^{2+\alpha}+\sqrt{\Delta}|D^2\psi|_0(M_X^2+M_Y^2).$$

(ii) If $\psi\in\mathcal{C}_{lb}^{\infty}(\rd)$, then
$$\mathcal{E}(\Delta,\psi)\leq \ \sqrt{\Delta}|D^3\psi|_0M_X^{3}+\sqrt{\Delta}|D^2\psi|_0(M_X^2+M_Y^2).$$

\end{proposition}
\begin{proof}
We only consider the case $d=1$, since the general case follows along similar albeit more complicated arguments. Note that for any $x\in\mathbb{R}$,
\begin{align*}
 &\ \mathbf{S}(\Delta)\psi(x)-\psi(x)-\Delta G(D\psi(x),D^2\psi(x)) \\
 \le &\
 \sube[\psi(x+\sqrt{\Delta}X+\Delta Y)-\psi(x)-\Delta D\psi(x)Y-\frac12 \Delta D^2\psi(x)X^2] \\
 \le &\
 \sube[\psi(x+\sqrt{\Delta}X)-\psi(x)-\frac12 \Delta D^2\psi(x)X^2]\\
 &\ +
 \sube[\psi(x+\sqrt{\Delta}X+\Delta Y)-\psi(x+\sqrt{\Delta}X)-\Delta D\psi(x)Y]:=(I)+(II).
\end{align*}

Next, we obtain upper bounds for terms $(I)$ and $(II)$. To this end, Taylor’s expansion and the assumption that $\sube[X]=\sube[-X]=0$ yield
\begin{align*}
(I) = &\
\sube\left[\sqrt{\Delta}D\psi(x)X+\int_x^{x+\sqrt{\Delta}X}\int_x^s(D^2\psi(u)-D^2\psi(x))duds\right] \\
= &\
\sube\left[\int_x^{x+\sqrt{\Delta}X}\int_x^s(D^2\psi(u)-D^2\psi(x))duds\right].
\end{align*}
In case (i), $|D^2\psi(u)-D^2\psi(x)|\le [D^2\psi]_{\mathcal{C}^\alpha}|u-x|^\alpha$, thus,
\begin{align*}
(I) \le &\
[D^2\psi]_{\mathcal{C}^\alpha}\sube\left[\int_x^{x+\sqrt{\Delta}X}\int_x^s|u-x|^\alpha du ds\right] \\
\le &\
[D^2\psi]_{\mathcal{C}^\alpha}\sube\left[\Delta^{1+\frac{\alpha}{2}}|X|^{2+\alpha}/(1+\alpha)(2+\alpha)\right] \le \Delta^{1+\frac{\alpha}{2}}[D^2\psi]_{\mathcal{C}^\alpha}M_X^{2+\alpha}.
\end{align*}
In case (ii), $|D^2\psi(u)-D^2\psi(x)|\le |D^3\psi|_0|u-x|$, thus,
\begin{align*}
(I) = &\
|D^3\psi|_0\sube\left[\int_x^{x+\sqrt{\Delta}X}\int_x^s|u-x|du ds\right] \\
\le &\
|D^3\psi|_0\sube\left[\Delta^\frac32|X|^3/6\right] \le \Delta^\frac32|D^3\psi|_0M_X^3.
\end{align*}
Regarding term (II), for both cases (i) and (ii), we have
\begin{align*}
 (II) = &\
\sube\left[ \int_{x+\sqrt{\Delta}X}^{x+\sqrt{\Delta}X+\Delta Y}(D\psi(s)-D\psi(x))ds\right]\\
= &\
\sube\left[ \int_{x+\sqrt{\Delta}X}^{x+\sqrt{\Delta}X+\Delta Y}\int_x^sD^2\psi(u)duds\right]\\
\le &\
|D^2\psi|_0\sube\left[\left|\frac{(\sqrt{\Delta}X+\Delta Y)^2-(\sqrt{\Delta}X)^2}{2}\right|\mathbf{1}_{\{\sqrt{\Delta}X(\sqrt{\Delta}X+\Delta Y)\ge 0\}}\right. \\
&\ +
\left.  \frac{(\sqrt{\Delta}X+\Delta Y)^2+(\sqrt{\Delta}X)^2}{2}\mathbf{1}_{\{\sqrt{\Delta}X(\sqrt{\Delta}X+\Delta Y)< 0\}}\right]\\
\le &\
|D^2\psi|_0\sube\left[(\Delta Y)^2/2+\Delta^\frac32|X||Y|\right] \\
\le &\
|D^2\psi|_0\sube\left[(\Delta Y)^2/2+\Delta^\frac32(|X|^2+|Y|^2)/2\right] \le
 \Delta^\frac32|D^2\psi|_0(M_X^2+M_Y^2)
\end{align*}
\textcolor[rgb]{1.00,0.00,0.00}{Combining the two estimates for terms (I) and (II)}, we obtain, for any $x\in\mathbb{R}$, that
$$\frac{\mathbf{S}(\Delta)\psi(x)-\psi(x)}{\Delta}-G(D\psi(x),D^2\psi(x))\le \Delta^{\alpha/2}[D^2\psi]_{\mathcal{C}^\alpha}M_X^{2+\alpha}+\sqrt{\Delta}|D^2\psi|_0(M_X^2+M_Y^2),$$
in case (i),
and that
$$\frac{\mathbf{S}(\Delta)\psi(x)-\psi(x)}{\Delta}-G(D\psi(x),D^2\psi(x))\le \sqrt{\Delta}|D^3\psi|_0M_X^{3}+\sqrt{\Delta}|D^2\psi|_0(M_X^2+M_Y^2),$$
in case (ii).
Similarly, we obtain lower bounds of $\mathcal{E}(\Delta,\psi)$, and this completes the proof.
\end{proof}

\subsection{The monotone approximation scheme}\label{sebsection: splitting}

For $\Delta\in(0,1)$, we let ${Q}_T^\Delta:=(\Delta,T]\times\rd$. Then, based on (\ref{scheme}) and $\mathbf{S}(\Delta)$, we introduce the approximation scheme as
\begin{eqnarray}\label{semischeme}
\left\{
\begin{array}{ll}
\displaystyle {S}(\Delta, x,u^\Delta(t,x),u^\Delta(t-\Delta,\cdot))=0 &\text{in}\  \bar{Q}_{T}^\Delta,\\
\displaystyle u^\Delta(t,x)=\phi(x)&\text{in}\
\bar{Q}_{T}\backslash\bar{Q}_{T}^\Delta,
\end{array}
\right.
\end{eqnarray}
where
${S}:(0,1)\times\rd\times\mathbb{R}\times\mathcal{C}_{lb}(\mathbb{R}^n)\rightarrow\mathbb{R}$ is defined by
\begin{equation}\label{Sbar2}
S(\Delta,x,p,v)=\frac{p-\mathbf{S}(\Delta)v(x)}{\Delta}.
\end{equation}

From the properties of the forward operator $\mathbf{S}(\Delta)$ and Proposition \ref{operator}, we obtain the following key
properties of the approximation scheme (\ref{semischeme}).

\begin{proposition}\label{scheme property}
Suppose that Assumption \ref{assumption1}(ii) is satisfied. Then,
the following properties hold for
the approximation scheme
${S}(\Delta,x,p,v)$ given in (\ref{semischeme}).

(i) (Monotonicity) For any $c_{1}, c_{2}\in\mathbb{R}$, and any function
$u\in\mathcal{C}_{lb}(\mathbb{R}^{n})$ with $u\le v$,
$${S}(\Delta,x,p+c_{1},u+c_{2})\ge {S}(\Delta,x,p,v)+\frac{c_{1}-c_{2}}{\Delta}.$$

(ii) (Concavity) $S(\Delta,x,p,v)$ is concave in $(p,v)$.

(iii) (Consistency) (a) If $\psi\in\mathcal{C}_{b}^{1+\frac{\alpha}{2},2+\alpha}({Q}_T)$  for some $\alpha\in(0,1)$, then
\begin{align}\label{consistant_error11}
 &\ |\partial_t\psi-G(D_{x}\psi,D_{x}^2\psi)-{S}(\Delta,x,\psi,\psi(t-\Delta,\cdot))|\notag\\
 \le  &\
 K_{\alpha}\left(\Delta^{\alpha/2}\left([D_x^2\psi]_{\mathcal{C}^{\alpha/2,\alpha}}+
 [\partial_{t}\psi]_{_{\mathcal{C}^{\alpha/2,\alpha}}}\right)+\sqrt{\Delta}|D_x^2\psi|_0+\Delta\left(|\partial_{t}D_x^2\psi|_0+|\partial_{t}D_x\psi|_{0}\right)\right)\ \text{in}\  {Q}_{T}^\Delta,
\end{align}
where the constant $K_{\alpha}$ is given by
\begin{equation}\label{constant_K_alpha}
K_{\alpha}:=1+M_Y^1+M_Y^2+M_X^2+M_X^{2+\alpha}.
\end{equation}

\indent (b) If $\psi\in\mathcal{C}_{lb}^{\infty}({Q}_T)$, then
\begin{align}\label{consistant_error}
 &\ |\partial_t\psi-G(D_{x}\psi,D_{x}^2\psi)-{S}(\Delta,x,\psi,\psi(t-\Delta,\cdot))|\notag\\
 \le  &\
 K_1\left(\sqrt{\Delta}\left(|D_x^3\psi|_0+|D_x^2\psi|_0\right)+\Delta\left(|\partial_{t}^2\psi|_{0}+|\partial_{t}D_x^2\psi|_0+|\partial_{t}D_x\psi|_{0}\right)\right)\ \text{in}\  {Q}_{T}^\Delta.
\end{align}
where the constant $K_1$ is given by
\begin{equation}\label{constant_K_1}
K_1:=1+M_Y^1+M_Y^2+M_X^2+M_X^{3}
\end{equation}
\end{proposition}

\begin{proof}
Parts (i)-(ii) are immediate, so
we only prove (iii). To this end, we split the consistency error into
three parts. Specifically, for $(t,x)\in {Q}_{T}^\Delta$,
\begin{align*}
&|\partial_t\psi-G(D_{x}\psi,D_{x}^2\psi)-{S}(\Delta,x,\psi,\psi(t-\Delta,\cdot))|\\
\leq &\ {\mathcal{E}(\Delta,\psi(t-\Delta,\cdot))} +
{|\psi(t,x)-\psi(t-\Delta,x)-\Delta\partial_t\psi(t,x)|\Delta^{-1}}\\
&+
{|G(D_{x}\psi(t,x),D_{x}^2\psi(t,x))-G(D_{x}\psi(t-\Delta,x),D_{x}^2\psi(t-\Delta,x))|}
:=(I)+(II)+(III),
\end{align*}
where $\mathcal{E}$ is defined in (\ref{operator_E}). Here we only consider the case (b); the case (a) only requires minor modification that is similar to the proof of Proposition \ref{operator}(i). For term (I),
Proposition \ref{operator} (ii) yields
\begin{equation}\label{estimate8}
\mathcal{E}(\Delta,\psi(t-\Delta,\cdot)) \leq (M_X^3+M_Y^2+M_X^2)\sqrt{\Delta}\left(|D_x^3\psi|_{0}+|D_x^2\psi|_{0}\right).
\end{equation}
For term (II), Taylor's expansion gives
\begin{align}\label{estimate9}
&\
{|\psi(t,x)-\psi(t-\Delta,x)-\Delta\partial_t\psi(t,x)|\Delta^{-1}}\notag\\
\leq&\
|\int_{t-\Delta}^{t}\left(\partial_{t}\psi(s,x)-\partial_{t}\psi(t,x)\right)ds|\Delta^{-1}\notag\\
\leq&\ \Delta^{-1}|\partial_{t}^2\psi|_0\int_{t-\Delta}^{t}\left(t-s\right)ds \le \Delta|\partial_{t}^2\psi|_0.
\end{align}
Finally, for term (III), we have
\begin{align}\label{estimate10}
&\
|G(D_{x}\psi(t,x),D_{x}^2\psi(t,x))-G(D_{x}\psi(t-\Delta,x),D_{x}^2\psi(t-\Delta,x))|\notag\\
\leq &\
\sube[|D_{x}\psi(t,x)-D_{x}\psi(t-\Delta,x)||Y|+\frac12|D_{x}^2\psi(t,x)-D_{x}^2\psi(t-\Delta,x)||X|^2]\notag\\
\leq &\ \Delta(M_Y^1|\partial_tD_{x}\psi|_{0}+M_X^2|\partial_{t}D_x^2\psi|_{0}),
\end{align}
Combining estimates (\ref{estimate8})-(\ref{estimate10}), we easily
conclude.
\end{proof}

\begin{remark}
Due to the monotonicity property (i) in Proposition \ref{scheme property}, the approximation scheme (\ref{semischeme}) is also referred to as the {monotone (approximation) scheme} in the sequel.
\end{remark}

\subsection{Comparison principle for the monotone approximation scheme}\label{sec-comparison}

The monotonicity property (i) in Proposition \ref{scheme property} also implies the following comparison principle for the monotone scheme (\ref{semischeme}), which will be used throughout this paper. Most of the arguments follow from Proposition 2.9 of \cite{Huang} (and Lemma 3.2 of \cite{BJ}), but we highlight some key
steps for the reader's convenience.

\begin{proposition}\label{schemecomparison}
Suppose that Assumption \ref{assumption1}(ii) is satisfied, and that $\underline{v}$,
$\bar{v}\in\mathcal{C}_{lb}(\bar{Q}_T)$ are such that
$${S}(\Delta,x,\underline{v},\underline{v}(t-\Delta,\cdot)) \leq h_1\ \text{in}\  {{Q}_{T}^\Delta},$$
$${S}(\Delta,x,\bar{v},\bar{v}(t-\Delta,\cdot)) \ge h_2\ \text{in}\  {{Q}_{T}^\Delta},$$
for some $h_1$, $h_2\in\mathcal{C}_{lb}({{Q}_{T}^\Delta})$. Then, for any $(t,x)\in \bar{Q}_{T}$.
\textcolor[rgb]{1.00,0.00,0.00}{\begin{equation}
\underline{v}(t,x)-\bar{v}(t,x)\leq \sup_{\bar{Q}_{T}\backslash
{{Q}_{T}^\Delta}}(\underline{v}(t,x)-\bar{v}(t,x))^{+}+t\sup_{{{Q}_{T}^\Delta}}(h_{1}(t,x)-h_{2}(t,x))^{+}.
\end{equation}
}
\end{proposition}

\begin{proof}
Without loss of generality, we assume that
\begin{equation}\label{comparisonforscheme}
\underline{v}\le \bar{v}\ \text{in}\ \bar{Q}_{T}\backslash {{Q}_{T}^\Delta}\ \
\text{and}\ \ h_{1}\leq h_{2}\ \text{in}\ {{Q}_{T}^\Delta},
\end{equation}
since, otherwise, the function $w:=\bar{v}+\sup_{\bar{Q}_{T}\backslash
{{Q}_{T}^\Delta}}(\underline{v}-\bar{v})^{+}+t\sup_{{Q}_{T}^\Delta}(h_{1}-h_{2})^{+}$
satisfies that $\underline{v}\le w$ in $\bar{Q}_{T}\backslash
{Q}_{T}^\Delta$ and by the monotonicity property (i) in Proposition
\ref{scheme property},
\begin{align*}
{S}(\Delta,x,w,w(t-\Delta,\cdot)) &\ge
 {S}(\Delta,x,\bar{v},\bar{v}(t-\Delta,\cdot))+\sup_{{Q}_{T}^\Delta}(h_{1}-h_{2})^{+}\\
&\ge h_{2}+\sup_{{Q}_{T}^\Delta}(h_{1}-h_{2})^{+} \ge \textcolor[rgb]{1.00,0.00,0.00}{h_{2}}\
\text{in}\ {Q}_{T}^\Delta.
\end{align*}
Thus, it suffices to prove $\underline{v}\le \bar{v}$ in $\bar{Q}_{T}$ when
(\ref{comparisonforscheme}) holds.

To this end, for $b\ge 0$, let $\psi_{b}(t):=bt$ and
$M(b):=\sup_{\bar{Q}_{T}}\{\underline{v}-\bar{v}-\psi_{b}\}.$ Then our aim is to
prove $M(0)\leq 0$ and we prove by contradiction. Assume $M(0)>0$, then by the continuity of $M$,
we must have $M(b)>0$ for some $b>0$. For such $b$, take a sequence
$\{(t_{n},x_{n})\}_{n\geq 1}$ in $\bar{Q}_{T}$ such that $\delta_{n}:=M(b)-(\underline{v}-\bar{v}-\psi_{b})(t_{n},x_{n})\downarrow 0,\ \ \text{as}\ n\to\infty.$
Since $M(b)>0$ but $\underline{v}-\bar{v}-\psi_{b}\leq 0$ in $\bar{Q}_{T}\backslash {Q}_{T}^\Delta$, we must have $t_{n}> \Delta$ for sufficiently large $n$. Then for such $n$, we use the monotonicity property (i) in Proposition \ref{scheme property} again to obtain
\begin{align*}
h_1(t_{n},x_{n}) \ge &\ {S}(\Delta,x_{n},\underline{v}(t_{n},x_{n}),\underline{v}(t_{n}-\Delta,\cdot)) \\
 \ge &\
 {S}(\Delta,x_{n},\bar{v}(t_{n},x_{n})+\psi_{b}(t_{n})+M(b)-\delta_{n},\bar{v}(t_{n}-\Delta,\cdot)+\psi_{b}(t_{n}-\Delta)+M(b)) \\
 \ge &\
 {S}(\Delta,x_{n},\bar{v}(t_{n},x_{n}),\bar{v}(t_{n}-\Delta,\cdot))+b-\delta_{n}\Delta^{-1} \\
 \ge &\
 h_{2}(t_{n},x_{n})+b-\delta_{n}\Delta^{-1},
\end{align*}
Since $h_{1}\leq h_{2}$ in ${Q}_{T}^\Delta$, we then must have
$b-\delta_{n}\Delta^{-1}\leq 0$. Thus, we deduce $b\leq 0$ by
letting $n\to\infty$, which is a contradiction.
\end{proof}

\subsection{Convergence of the monotone approximation scheme}\label{sec-convergence}

We prove Theorem \ref{maintheorem}(i) by showing the convergence of the approximate solution $u^{\Delta}$ to the viscosity solution $u$. It is based on the monotone schemes for viscosity solutions introduced by Barles-Souganidis in \cite{Barles}, where they
show that any \emph{monotone, stable} and \emph{consistent}
numerical scheme converges, provided that there exists a comparison principle for the limiting equation.

To this end, define the semi-relaxed limits of $u^{\Delta}$ by
$$\overline{u}(t,x)=\limsup_{(t',{x}')\rightarrow(t,x),
\atop \Delta\rightarrow 0}u^{\Delta}(t',x');\ \ \
\underline{u}(t,x)=\liminf_{(t',{x}')\rightarrow(t,x),
\atop \Delta\rightarrow 0}u^{\Delta}(t',x').
$$
We show that $\overline{u}$ is a viscosity subsolution of
(\ref{PDE_1})-(\ref{cc}). A symmetric argument will imply that $\underline{u}$
is a viscosity supersolution of (\ref{PDE_1}), which proves that
$\overline{u}=\underline{u}=u$, so $u^{\Delta}$ converges to $u$
locally uniformly.

Let $\phi\in\mathcal{C}^{\infty}(\bar{Q}_T)$ and $(t_0,x_0)\in Q_T$
be such that
$$0=(\overline{u}-\phi)(t_0,x_0)=\max_{(t',x')}
(\overline{u}-\phi)(t',x').$$ By the definition of $\overline{u}$,
there exists a sequence $\{(t_n,x_n,\Delta_{n})\}_{n\geq 1}$ such that
$$(t_n,x_n,\Delta_{n})\rightarrow(t_0,x_0,0),\ \ \text{and}\ \ u^{\Delta_{n}}(t_n,x_n)\rightarrow\overline{u}(t_0,x_0).$$
Moreover, by extracting a subsequence if necessary, $(t_n,x_n)$ is
also the maximum point of $u^{\Delta_{n}}-\phi$:
$$\delta^{\Delta_{n}}:=(u^{\Delta_{n}}-\phi)(t_n,x_n)=\max_{(t',x')}
(u^{\Delta_{n}}-\phi)(t',x')\rightarrow 0.$$
Since $t_{0}>0$ and $\Delta\rightarrow 0$, we have $t_{n}>\Delta_{n}$ for large enough $n$. The monotonicity property (i) in Proposition
\ref{scheme property} further implies that
\begin{align*}
0&=S(\Delta_{n},x_{n},u^{\Delta_{n}}(t_{n},x_{n}),u^{\Delta_{n}}(t_{n}-\Delta_{n},\cdot))\\
  \ge&\ S(\Delta_{n},x_{n},\phi(t_{n},x_{n})+\delta^{\Delta_{n}},\phi(t_{n}-\Delta_{n},\cdot)+\delta^{\Delta_{n}})\\
  =&\ \frac{\phi(t_{n},x_{n})-\mathbf{S}(\Delta_{n})\phi(t_{n}-\Delta_{n},\cdot)(x_{n})}{\Delta_{n}}.
\end{align*}
In turn, using the consistency property (iii) in Proposition  \ref{scheme property} and letting
$(t_n,x_n,\Delta)\rightarrow(t_0,x_0,0)$, we obtain
$$\partial_t\phi(t_0,x_0)-
G(D_x\phi(t_0,x_0),D_{x}^2\phi(t_0,x_0))\leq 0.$$

Next, we show that $\overline{u}(0,x)=\phi(x)$ for $x\in\mathbb{R}^d$. Let $\{(t_n,x_n, \Delta_n)\}_{n\geq 1}$ be a sequence such that
$$(t_n,x_n,\Delta_{n})\rightarrow(0,x,0),\ \ \text{and}\ \ u^{\Delta_{n}}(t_n,x_n)\rightarrow\overline{u}(0,x).$$
Since $u^{\Delta_n}(s,x_n)=\phi(x_n)$ for $s\in\bar{Q}_{T}\backslash \bar{Q}_{T}^{\Delta_n}$,
by the time regularity of $u^{\Delta}$ in Lemma \ref{regularity1}(ii), we have
$$|u^{\Delta_n}(t_n,x_n)-u^{\Delta_n}(s,x_n)|\le \sqrt{3}C_\phi K_0(|t_n-s|^{\beta/2}+\Delta_n^{\beta/2}).$$
{Letting $s=0$ and} sending $n\rightarrow\infty$ {yield} that $|\bar{u}(0,x)-\phi(x)|=0$, from which we conclude that
$\overline{u}(\cdot,\cdot)$ is a viscosity subsolution of
(\ref{PDE_1})-(\ref{cc}).


\section{Convergence rate of the monotone approximation  scheme}\label{sec-main}

In this section, we prove Theorem \ref{maintheorem}(ii) by
establishing the (uniform) convergence rate of the approximate solution
$u^{\Delta}$ to the viscosity solution $u$, and keeping track of all the involved constants.
We start with the approximation error in the first time interval $\bar{Q}_{T}\backslash {Q}_{T}^\Delta$, where $u^{\Delta}=\phi = u|_{t=0}$ except at $t=\Delta$. Therefore, the bound for the approximation error in this interval can be easily obtained by the regularity property of $u$ in Lemmas  \ref{regularity}. This is demonstrated in the following lemma.

\begin{lemma}\label{errorsmall}
Suppose that Assumption \ref{assumption1} is satisfied.
Then, for $\Delta\in(0,1)$,
\begin{equation}\label{estimate_final_interval}
|u-u^{\Delta}|\leq 2C_{\phi}K_0{\Delta}^{\beta/2}\ in\ \bar{Q}_{T}\backslash {Q}_{T}^\Delta,
\end{equation}
where the constant $K_0$ is given in (\ref{constant_K_0}).
\end{lemma}

\begin{proof}
Since $u^{\Delta}=\phi = u|_{t=0}$ in $\bar{Q}_{T}\backslash \bar{Q}_{T}^\Delta$, we have, for $(t,x)\in\bar{Q}_{T}\backslash {Q}_{T}^\Delta$,
$$|u(t,x)-u^\Delta(t,x)|\le |u(t,x)-u(0,x)|+|u^\Delta(\Delta,x)-u^\Delta(0,x)|\mathbf{1}_{\{t=\Delta\}}.$$
When $t=\Delta$, we further obtain
\begin{align*}
 |u^\Delta(\Delta,x)-u^\Delta(0,x)| \le
 &\
 \sube[|u^\Delta(0,x+\sqrt{\Delta}X+\Delta Y)-u^\Delta(0,x)|] \\
 =&\ \sube[|\phi(x+\sqrt{\Delta}X+\Delta Y)-\phi(x)|] \\
 \le &\
 \sube[C_\phi |\sqrt{\Delta}X+\Delta Y|^\beta] \\
 \le &\
 C_\phi(M_X^\beta+\Delta^{\beta/2}M_Y^\beta)\Delta^{\beta/2} \\
  \le &\
 C_\phi((M_X^2)^{\beta/2}+(M_Y^1)^\beta)\Delta^{\beta/2}\leq C_{\phi}K_0\Delta^{\beta/2}.
\end{align*}
The conclusion then follows from Lemma \ref{regularity}(ii).
\end{proof}


\subsection{Lower bound for the approximation {error}}\label{upbd}

For $u\in\mathcal{C}_{lb}^{\frac{\beta}{2},\beta}(\bar{Q}_T)$, we aim to derive a lower bound for the approximation error $u-u^{\Delta}$ within the whole domain $\bar{Q}_{T}$.
To this end, for $\varepsilon\in(0,1)$, we extend the domain of the G-equation (\ref{PDE_1}) from $Q_T$ to
${Q}_{T+\varepsilon^{2}}:=(0,T+\varepsilon^{2}]\times\mathbb{R}^{d}$
and still denote the solution as $u$.
Next, we  regularize $u$ by
a standard mollification procedure: let $\rho(t,x)$ be a nonnegative
smooth function with support in $(-1,0)\times B(0,1)$
and mass $1$, and introduce the sequence of mollifiers
$\rho_{\varepsilon}$ for $\varepsilon\in(0,1)$,
\begin{equation}\label{mollifer}
\rho_{\varepsilon}(t,x):=\frac{1}{\varepsilon^{2+d}}\rho\left(\frac{t}{\varepsilon^2},\frac{x}{\varepsilon}\right).
\end{equation}
For $(t,x)\in \bar{Q}_T$, we then define
$$u_{\varepsilon}(t,x)=u*
\rho_{\varepsilon}(t,x)=\int_{-\varepsilon^2< \tau< 0}\int_{|e|<
\varepsilon}u(t-\tau,x-e)\rho_{\varepsilon}(\tau,e)ded\tau.$$
Lemma \ref{regularity} implies that  
$$|u(t,x)-u(s,y)|\le C_\phi\left[|x-y|^\beta+K_0|s-t|^{\beta/2}\right].$$
In turn, standard properties of mollifiers (see C.4 in \cite{Evans}) imply that
$u_{\varepsilon}\in\mathcal{C}_{lb}^{\infty}(\bar{Q}_{T})$,
\begin{equation}\label{lowerbound_2}
|u-u_{\varepsilon}|_0\leq C_{\phi}(1+K_0)\varepsilon^\beta,
\end{equation}
and, moreover, for positive integer $i$ and multiindex $j$,
\begin{equation}\label{mollifier}
|\partial_{t}^iD_{x}^ju_{\varepsilon}|_0\leq C_{\phi}(1+K_0)\varepsilon^{\beta-2i-|j|}||\partial_{t}^iD_{x}^j\rho||_1,
\end{equation}
where the constant $K_0$ is given in (\ref{constant_K_0}) and
\textcolor[rgb]{1.00,0.00,0.00}{$$||\partial_{t}^iD_{x}^j\rho||_1=\int_{(-1,0)}\int_{B(0,1)}\left|\partial_{t}^iD_{x}^j\rho(\tau,e)\right|d(\tau,e)<\infty.$$}

We observe that the function $u(t-\tau,x-e)$ is still a viscosity solution of the G-equation (\ref{PDE_1})
in ${Q}_{T}$ for any
{$(\tau,e)\in(-\varepsilon^2,0)\times B(0,\varepsilon)$}. On the other hand, a Riemann sum approximation shows that there exists a sequence $\{I_n\}_{n\ge1}\in\mathcal{C}_{lb}(\bar{Q}_T)$ such that each $I_n$ is a convex combination of the functions $u(\cdot-\tau,\cdot-e)$ for different $(\tau,e)\in(-\varepsilon^2,0)\times B(0,\varepsilon)$ and that $I_n$ converges uniformly to $u_\varepsilon$. Since the nonlinear term
$G(p,X)$ is convex in $p$ and
$X$, each $I_n$ becomes a supersolution of (\ref{PDE_1}) in ${Q}_{T}$. Using the stability of viscosity solutions, we deduce that $u_{\varepsilon}(t,x)$ is still a supersolution
of (\ref{PDE_1}) in $Q_T$, namely,
\begin{equation}\label{sub2}
\partial_{t}u_{\varepsilon}-G(D_{x}u_{\varepsilon},D_{x}^2u_{\varepsilon})\ge 0.
\end{equation}
%
%
%

We are now in a position to establish a lower bound for the approximation error.

\begin{theorem}\label{theorem_error_1}
Suppose that Assumption \ref{assumption1} is  satisfied. Then, for $\Delta\in(0,1)$, there exists a constant $C_{LB}$ depending only on $T$, $C_{\phi}$, $\beta$, $M_X^{3}$ and $M_Y^2$ such that
$$u-u^{\Delta}\ge -C_{LB}\Delta^{\beta/6}\ \text{in}\ \bar{Q}_{T}.$$
Moreover, the constant $C_{LB}$ has an explicit formula
$C_{LB}:=C_{\phi}(1+K_0)\left(4+K_1C_\rho T\right)$ with the constants $K_0$, $K_1$ and $C_{\rho}$ given in  (\ref{constant_K_0}),  (\ref{constant_K_1}) and (\ref{C_rho}), respectively.
\end{theorem}

\begin{proof}
Since $u_{\varepsilon}\in\mathcal{C}_{lb}^{\infty}(\bar{Q}_{T})$ is smooth with bounded derivatives of any order, we substitute $u_{\varepsilon}$ into the consistency error estimate
(\ref{consistant_error}) and use (\ref{sub2}) and (\ref{mollifier}) to obtain
\begin{align}\label{estimate4}
&\ {S}(\Delta,x,u_{\varepsilon}(t,x),u_{\varepsilon}(t-\Delta,\cdot)) \nonumber \\
 \ge &\
-C_{\phi}(1+K_0)K_1\nonumber\\
&\ \times\left[\sqrt{\Delta}(\varepsilon^{\beta-3}||D^3_x\rho||_1+\varepsilon^{\beta-2}||D^2_x\rho||_1)+\Delta(\varepsilon^{\beta-4}(||\partial^2_t\rho||_1+||\partial_tD^2_x\rho||_1)+\varepsilon^{\beta-3}||\partial_tD_x\rho||_1)\right] \nonumber\\
\ge  &\ -C_{\phi}(1+K_0)K_1\nonumber\\
&\ \times\left[\sqrt{\Delta}\varepsilon^{\beta-3}(||D^3_x\rho||_1+||D^2_x\rho||_1)+\Delta\varepsilon^{\beta-4}(||\partial^2_t\rho||_1+||\partial_tD^2_x\rho||_1+||\partial_tD_x\rho||_1)\right]\nonumber\\
=: &\ -C_{\phi}(1+K_0)K_1c(\beta,\varepsilon),
\end{align}
for $(t,x)\in{{Q}_{T}^\Delta}$, where the constants $K_0$ and $K_1$ are given in (\ref{constant_K_0}) and (\ref{consistant_error}), respectively.
The comparison principle
in Proposition \ref{schemecomparison} then implies that in $\bar{Q}_T$,
$$u^{\Delta}-u_{\varepsilon}\leq \sup_{\bar{Q}_{T}\backslash {{Q}_{T}^\Delta}}(u^{\Delta}-u_{\varepsilon})^{+}+c(\beta,\varepsilon)TC_{\phi}(1+K_0)K_1.$$
Next, using (\ref{lowerbound_2}), we
further obtain
\begin{align*}
u^{\Delta}-u &=(u_{\varepsilon}-u)+(u^{\Delta}-u_{\varepsilon})\\
&\leq C_{\phi}(1+K_0)\varepsilon^\beta+\sup_{\bar{Q}_{T}\backslash
{{Q}_{T}^\Delta}}(u^{\Delta}-u_{\varepsilon})^{+}+c(\beta,\varepsilon)TC_{\phi}(1+K_0)K_1\\
&\leq \sup_{\bar{Q}_{T}\backslash
{{Q}_{T}^\Delta}}(u-u^{\Delta})^{+}+2C_{\phi}(1+K_0)\varepsilon^\beta+c(\beta,\varepsilon)TC_{\phi}(1+K_0)K_1 \
\text{in}\ \bar{Q}_T.
\end{align*}
By choosing $\varepsilon=\Delta^{1/6}$, we conclude that
\begin{align*}
u^{\Delta}-u&\ \leq \sup_{\bar{Q}_{T}\backslash {{Q}_{T}^\Delta}}(u-u^{\Delta})^{+}+2C_{\phi}(1+K_0)\Delta^{\beta/6}+c(\beta,\Delta^{1/6})TC_{\phi}(1+K_0)K_1 \\
&\ \le C_{\phi}(1+K_0)\left(4+K_1C_\rho T\right)\Delta^{\beta/6}\ \ \text{in}\ \bar{Q}_T,
\end{align*}
where the last inequality follows from the estimate
(\ref{estimate_final_interval}) in Lemma \ref{errorsmall} and the fact that $c(\beta,\Delta^{1/6})\leq C_{\rho}\Delta^{\beta/6}$ with
\begin{equation}\label{C_rho}
C_\rho: = ||D^3_x\rho||_1+||D^2_x\rho||_1+||\partial^2_t\rho||_1+||\partial_tD^2_x\rho||_1+||\partial_tD_x\rho||_1<\infty.
\end{equation}
\end{proof}
\begin{remark}
A typical example of $\rho$ is given by
$$\rho(t,x)=K\exp(-\frac{1}{1-|x|^2})\exp(-\frac{1}{1-(2t+1)^2})\mathbf{1}_{\{|x|<1,-1<t<0\}},$$
where $K$ is given such that the mass of $\rho$ is 1. One can always compute $C_\rho$ using this example and in the one-dimension case, $C_\rho<10^3e^{-1}$. In turn, when $T=1$, it follows from the formulae for $K_0$ and $K_1$ (c.f. (\ref{constant_K_0}) and (\ref{constant_K_1})) that
\begin{align*}
C_{LB}&\leq C_{\phi}\left[1+e^{\frac{\beta}{2}}((M_X^2)^{\frac{\beta}{2}}+(M_Y^2)^{\frac{\beta}{2}})\right]\left[4+(1+M_Y^1+M_Y^2+M_X^2+M_X^{3})\frac{10^3}{e} \right]\\
&\leq 613C_{\phi}\left[1+(M_X^3)^{\frac{\beta}{3}}+(M_Y^2)^{\frac{\beta}{2}}\right]
\left[1+(M_X^3)^{\frac23}+M_X^3+(M_Y^2)^{\frac12}+M_Y^2\right].
\end{align*}
\end{remark}

\subsection{Upper bound for the approximation {error}}\label{lowbd}

To obtain an upper bound for the approximation error, we are not able to construct approximate smooth subsolutions {of} (\ref{PDE_1}) due to the
convexity of the function $G$. Instead, we interchange the
roles of the G-equation (\ref{PDE_1}) and the monotone scheme
(\ref{semischeme}) (as in \cite{HL1} and \cite{Jakobsen}).

To this end, for $\varepsilon\in(0,1)$, we extend the domain of the monotone scheme (\ref{semischeme}) from $\bar{Q}_T$ to
$\bar{Q}_{T+\varepsilon^{2}}:=[0,T+\varepsilon^{2}]\times\mathbb{R}^{d}$
and still denote the scheme solution as $u^\Delta$.
Then, using the same mollifiers $\rho_\varepsilon$ as in section \ref{upbd}, we define, for $(t,x)\in\bar{Q}_T$,
$$u_{\varepsilon}^\Delta(t,x)=u^\Delta*
\rho_{\varepsilon}(t,x)=\int_{-\varepsilon^2< \tau< 0}\int_{|e|<
\varepsilon}u^\Delta(t-\tau,x-e)\rho_{\varepsilon}(\tau,e)ded\tau.$$
The regularity property of $u^\Delta$ in Lemma \ref{regularity1} implies that
$$|u^{\Delta}(t,x)-u^{\Delta}(s,y)|\leq C_{\phi}\left[|x-y|^{\beta}+\sqrt{3}K_0(|s-t|^{\beta/2}+\Delta^{\beta/2})\right].$$
In turn, standard properties of mollifiers imply that $u_{\varepsilon}^\Delta\in\mathcal{C}_{lb}^{\infty}(\bar{Q}_{T})$,
\begin{equation}\label{upperbound_2}
|u^\Delta-u^\Delta_{\varepsilon}|_0\leq C_\phi(1+\sqrt{3}K_0)(\varepsilon^\beta+\Delta^{\beta/2}),
\end{equation}
and, moreover, for positive integer $i$ and multiindex $j$,
\begin{equation}\label{mollifier1}
|\partial_{t}^iD_{x}^ju^\Delta_{\varepsilon}|_0\leq C_\phi(1+\sqrt{3}K_0)\varepsilon^{-2i-|j|}(\varepsilon^\beta+\Delta^{\beta/2})||\partial_{t}^iD_{x}^j\rho||_1,
\end{equation}
where the constant $K_0$ is given in (\ref{constant_K_0}) and
\textcolor[rgb]{1.00,0.00,0.00}{$$||\partial_{t}^iD_{x}^j\rho||_1=\int_{(-1,0)}\int_{B(0,1)}\left|\partial_{t}^iD_{x}^j\rho(\tau,e)\right|d(\tau,e)<\infty.$$}

Next, let $\{I_n^\Delta\}_{n\ge1}\in\mathcal{C}_{lb}(\bar{Q}_T)$ be a sequence such that each $I_n^\Delta$ is a convex combination of the functions $u^\Delta(\cdot-\tau,\cdot-e)$ for different $(\tau,e)\in(-\varepsilon^2,0)\times B(0,\varepsilon)$ and that $I_n^\Delta$ converges uniformly to $u_\varepsilon^\Delta$.
Since $$S(\Delta,x,u^\Delta(t-\tau,x-e),u^\Delta(t-\tau-\Delta,\cdot-e))=0\ \text{in}\ \bar{Q}_T^\Delta,$$ for any
{$(\tau,e)\in(-\varepsilon^2,0)\times B(0,\varepsilon)$}, the concavity of the monotone scheme (cf. Proposition \ref{scheme property} (ii)) yields that for any $n\in\mathbb{N}$ and $(t,x)\in\bar{Q}_T^\Delta$,
$$S(\Delta, x, I_n^\Delta(t,x),I_n^\Delta(t-\Delta,\cdot))\ge 0.$$
Since $I_n^\Delta$ is lower bounded, we use Fatou's property of the sublinear expectation $\sube$ (see (\ref{Fatou})) to deduce that, for $(t,x)\in\bar{Q}_T^\Delta$,
\begin{align}\label{super2}
&\ S(\Delta, x, u_{\varepsilon}^\Delta(t,x),u_{\varepsilon}^\Delta(t-\Delta,\cdot)) \nonumber \\
 = &\
 \left(u_{\varepsilon}^\Delta(t,x)-\sube[u_{\varepsilon}^\Delta(t-\Delta,x+\sqrt{\Delta}X+\Delta Y)]\right)\Delta^{-1} \nonumber \\
 \ge &\
  \left(u_{\varepsilon}^\Delta(t,x)-\lim_{n\to\infty}\sube[I_n^\Delta(t-\Delta,x+\sqrt{\Delta}X+\Delta Y)]\right)\Delta^{-1} \nonumber \\
  = &\
  \lim_{n\to\infty}S(\Delta, x, I_n^\Delta(t,x),I_n^\Delta(t-\Delta,\cdot))\ge 0.
\end{align}

We are now in a position to establish an upper bound for the approximation error.

\begin{theorem}\label{theorem_error_111}
\textcolor[rgb]{1.00,0.00,0.00}{Suppose that Assumption \ref{assumption1} is satisfied.} Then, for $\Delta\in(0,1)$, there exists a constant $C_{UB}$ depending only on $T$, $C_{\phi}$, $\beta$, $M_X^{3}$ and $M_Y^2$ such that
$$u-u^{\Delta}\le C_{UB}\Delta^{\beta/6}\ \text{in}\ \bar{Q}_{T}.$$
Moreover, the constant $C_{UB}$ has an explicit formula
$C_{UB}:=2\sqrt{3}C_{LB}=2\sqrt{3}C_{\phi}(1+K_0)\left(4+K_1C_\rho T\right)$ with the constants $K_0$, $K_1$ and $C_{\rho}$ given in  (\ref{constant_K_0}),  (\ref{consistant_error}) and (\ref{C_rho}), respectively.
\end{theorem}

\begin{proof}
We first consider the above error estimate in $Q^\Delta_T$. Since $u_{\varepsilon}^\Delta\in\mathcal{C}_{lb}^{\infty}(\bar{Q}_{T})$ is smooth with bounded derivatives of any order, we substitute $u_{\varepsilon}^\Delta$ into the consistency error estimate
(\ref{consistant_error}) and use (\ref{super2}) and  (\ref{mollifier1}) to obtain
\begin{align*}\label{estimate44}
\partial_{t}u_{\varepsilon}^\Delta-G(D_{x}u_{\varepsilon}^\Delta,D_{x}^2u_{\varepsilon}^\Delta)
 \ge
-C_{\phi}(1+\sqrt{3}K_0)K_1(\varepsilon^\beta+\Delta^{\beta/2})c(0,\varepsilon)
\end{align*}
for $(t,x)\in{{Q}_{T}^\Delta}$, where $c(0,\varepsilon)$ is defined in (\ref{estimate4}).
Then, the function $$\bar{v}(t,x):=u_{\varepsilon}^\Delta(t,x)+C_{\phi}(1+\sqrt{3}K_0)K_1(\varepsilon^\beta+\Delta^{\beta/2})c(0,\varepsilon)(t-\Delta)$$ becomes a (classical) supersolution of the G-equation (\ref{PDE_1}) {in $Q_T^\Delta$} with {initial condition $\bar{v}(\Delta,x)=u_{\varepsilon}^\Delta(\Delta,x)$}. On the other hand, from (\ref{upperbound_2}) and (\ref{errorsmall}), we know that $$\underline{v}(t,x):=u(t,x)-C_\phi(1+\sqrt{3}K_0)(\varepsilon^\beta+\Delta^{\beta/2})-
2C_{\phi}K_0{\Delta}^{\beta/2}$$ is a (viscosity) solution of the G-equation (\ref{PDE_1}), and from  (\ref{estimate_final_interval}) and (\ref{upperbound_2}), we further have
\begin{align*}
\underline{v}(\Delta,x)=&\ u(\Delta,x)-C_\phi(1+\sqrt{3}K_0)(\varepsilon^\beta+\Delta^{\beta/2})-
2C_{\phi}K_0{\Delta}^{\beta/2}\\
=&\ u(\Delta,x)-u^{\Delta}(\Delta,x)+u^{\Delta}(\Delta,x)-u_{\varepsilon}^\Delta(\Delta,x)
+u_{\varepsilon}^\Delta(\Delta,x)\\
&-
2C_{\phi}K_0{\Delta}^{\beta/2}-C_\phi(1+\sqrt{3}K_0)(\varepsilon^\beta+\Delta^{\beta/2})\leq
u_{\varepsilon}^\Delta(\Delta,x)=\bar{v}(\Delta,x).
\end{align*}
Thus, the comparison principle for the G-equation (see Theorem 6.3 in \cite{Peng2008}) implies that $\underline{v}\le\bar{v}$ in {$\bar{Q}^\Delta_T$}, i.e.
$$u-u_{\varepsilon}^\Delta\leq
C_\phi(1+\sqrt{3}K_0)(\varepsilon^\beta+\Delta^{\beta/2})+
2C_{\phi}K_0{\Delta}^{\beta/2}
+C_{\phi}(1+\sqrt{3}K_0)K_1(\varepsilon^\beta+\Delta^{\beta/2})c(0,\varepsilon)(t-\Delta)\ \text{in}\ {\bar{Q}^\Delta_T}.$$
Finally, using the estimates (\ref{upperbound_2}) again, we obtain by choosing $\varepsilon=\Delta^{1/6}$ that
\begin{align*}
u-u^{\Delta} &=(u-u_{\varepsilon}^\Delta)+(u_{\varepsilon}^\Delta-u^{\Delta})\\
&\leq 4C_\phi(1+\sqrt{3}K_0)\Delta^{\beta/6}+2C_{\phi}K_0{\Delta}^{\beta/6}
+2C_{\phi}(1+\sqrt{3}K_0)K_1C_{\rho}T\Delta^{\beta/6}\
\text{in}\ {\bar{Q}^\Delta_T},
\end{align*}
where we used the fact that $c(0,\Delta^{1/6})\leq C_{\rho}$. The conclusion then follows by combining the above estimate with (\ref{estimate_final_interval}).
\end{proof}

\subsection{The non-degenerate case}\label{sec-special}


We prove part(iii) in Theorem \ref{maintheorem}.  When the  non-degeneracy assumption and more regularity on the initial data $\phi$ are
imposed as in part(iii), the solution $u$ of (\ref{PDE_1})-(\ref{cc}) becomes a classical solution with enough regularity. This will significantly simplify the previous proof for the general case with possible degeneracy.

%


First, the monotonicity property of \textcolor[rgb]{1.00,0.00,0.00}{$\widetilde{\mathbb{E}}$}, the boundedness of $\phi$ and (\ref{solution}) yield that $u$ is bounded. Lemma \ref{regularity} further implies that $u\in\mathcal{C}^{{1}/{2},1}_b(\bar{Q}_T)$. In turn, the regularity theory of fully nonlinear PDEs implies the H\"older continuity of the derivatives of $u$, i.e.  there exists a constant $\alpha\in(0,1)$ {depending only on $d$, $\underline{\sigma}^2$ and $M_X^2$} such that   $u\in\mathcal{C}_b^{1+\frac{\alpha}{2},2+\alpha}(\bar{Q}^{\varepsilon}_{T})$
for any $\varepsilon>0$ (see Theorem 4.5 in Appendix C of \cite{Peng2010}, or  \cite{Krylov3} and \cite{Wang} for more details). The consistency error
estimate (\ref{consistant_error11}) then yields
\begin{align*}
&|{S}(\Delta,x,u(t,x),u(t-\Delta,\cdot))|\\
\leq&\
K_{\alpha}\left(\Delta^{\alpha/2}\left([D_x^2u]_{\mathcal{C}^{\alpha/2,\alpha}}+[\partial_{t}u]_{{\mathcal{C}^{\alpha/2,\alpha}}}\right)+\sqrt{\Delta}|D_x^2u|_0+\Delta\left(|\partial_{t}D_x^2u|_0+|\partial_{t}D_xu|_{0}\right)\right)\leq C\Delta^{\alpha/2},
\end{align*}
 for
$(t,x)\in{Q}_{T}^\Delta$ and some constant $C$.
 On the other hand, since
 $${S}(\Delta,x,u^{\Delta}(t,x),u^{\Delta}(t-\Delta,\cdot))=0,$$
the comparison principle
in Proposition \ref{schemecomparison} implies
\begin{equation}\label{non-degenerate-case}
|u-u^{\Delta}|\leq \sup_{\bar{Q}_T\backslash{Q}_{T}^\Delta}|u-u^{\Delta}|+Ct\Delta^{\alpha/2}\ \ \text{in}\ \bar{Q}_T.
\end{equation}
Since Assumption \ref{assumption1}(i) holds with $\beta=1$, it follows from Lemma \ref{errorsmall}
that
$$\sup_{\bar{Q}_T\backslash{Q}_{T}^\Delta}|u-u^{\Delta}|\leq 2C_{\phi}K_0\Delta^{1/2}.$$
The conclusion follows by plugging the above estimate into  (\ref{non-degenerate-case}) and combining with part(ii) in Theorem \ref{maintheorem}.

\begin{remark}
Since there is no explicit formula for the H\"older constant $\alpha$, we are not able to write down the explicit error bound as for the general case with possible degeneracy in part(ii) of Theorem \ref{maintheorem}.

On the other hand, if the solution $u$ has more regularity, say $u\in\mathcal{C}_b^{\infty}({Q}_T)$, then we can replace the consistency error estimate (\ref{consistant_error11}) in the above proof  by (\ref{consistant_error}), and obtain the convergence rate $\Delta^{1/2}$.
\end{remark}

\section{Some special cases}
In this section, we improve the convergence rates in Theorem \ref{Theorem_CLT} by imposing further model assumptions, and compare our results with the existing literature. For the latter use, we state the following property (see Proposition 4.1 in \cite{Peng2008}) of the nonlinear function $G(p,A)$ given by (\ref{functionG}).

\begin{proposition}\label{bddsets}
Let the nonlinear function $G(p,A)$ be given in (\ref{functionG}). Then, there exists a bounded, closed and compact subset $\Theta\subset\rd\times\mathbb{R}^{d\times d}$ such that
$$G(p,A)=\sup_{(q,Q)\in\Theta}\left\{\frac12 tr[AQQ^T]+\langle p,q\rangle\right\},\ (p,A)\in\rd\times\mathbb{S}(d).$$
\end{proposition}

\subsection{Law of large numbers: Comparison with \cite{FPSS}}

Assume that $X=0$.
With this extra assumption, we can obtain a better convergence rate by refining the consistency error estimates in Proposition \ref{operator} and Proposition \ref{scheme property}.

\begin{corollary}\label{corollary1} Suppose that Assumption \ref{assumption1} is satisfied with $X=0$ and $\beta=1$, i.e. there is no volatility uncertainty, and the initial data $\phi$ is Lipschitz continuous bounded from below. Then, there exists a constant $C$ depending only on $T$, ${\Theta}$  and $M_Y^2$ such that
$$|u-u^{\Delta}|\leq C\Delta^{\frac12}\ in\ \bar{Q}_T.$$
\end{corollary}

Before proving Corollary \ref{corollary1}, we show its application to the law of large numbers. To this end, let $\Delta=\frac1n$, then by the representation formula (\ref{clt_representation}), we have  $$u^{1/n}(1,0)=\sube[\phi(\sum_{i=1}^n \frac{Y_i}{n})],$$ where $Y_1=Y$, $Y_{i+1} \stackrel{d}{=} Y_i$ and $Y_{i+1}$ is independent of $(Y_1,...,Y_i)$ for each $i=1,...,n-1$.
On the other hand, if we further let $\phi(y):=d_{\Theta}(y)=\inf\{|x-y|:x\in\Theta\},$ where the subset $\Theta\subset\rd$ is given in $G(p,0)$ in Proposition \ref{bddsets}, then $d_{\Theta}(y)\ge 0$ is Lipschitz continuous bounded from below. It follows from Example 4.3 in \cite{Peng2008} that  $$u(1,0)=\sup_{\theta\in\Theta}\phi(\theta)=\sup_{\theta\in\Theta}d_{\Theta}(\theta)=0.$$ In turn, Corollary \ref{corollary1} yields the following form of \emph{law of large numbers}
\begin{equation}\label{LLN}
0\leq\sube[d_{\Theta}(\sum_{i=1}^n \frac{Y_i}{n})]\leq Cn^{-1/2}.
\end{equation}
Note that the above convergence rate is better than the convergence rate $n^{-2/5}$ in Fang et al \cite{FPSS} for the law of large numbers under sublinear expectations (See Remark 2.3 in \cite{FPSS}).

\begin{remark}\footnote{While revising
this work, we came across the working paper \cite{Song2} by Song posted on arXiv in April 2019. He also considers this special case, and obtains the convergence rate $n^{-1/2}$ using the Stein's method and a mollification procedure.} If we choose  any bounded and Lipschitz continuous $\phi\in\mathcal{C}_b^1(\mathbb{R}^d)$ as the test function (which clearly satisfies the assumption in Corollary \ref{corollary1}), then we also obtain the following general form of law of large numbers
\begin{equation}\label{LLN_2}
\left|\sube[\phi(\sum_{i=1}^n \frac{Y_i}{n})]-\sup_{\theta\in\Theta}\phi(\theta)\right|\leq Cn^{-1/2}.
\end{equation}
\end{remark}

We proceed to prove Corollary \ref{corollary1}. Since it is a special case of Theorem \ref{maintheorem}, we only highlight its main steps and differences compared to the proof of Theorem \ref{maintheorem}. Unless otherwise specified, $C$ will represent a generic constant in the following.

\emph{Step 1}. Since $X=0$ and $\beta=1$, a revisit of Lemmas \ref{regularity} and \ref{regularity1} shows that $u$ and $u^{\Delta}$ satisfy
\begin{align*}
  |u(t,x)-u(s,y)| & \leq C_{\phi}[|x-y|+M_Y^1|s-t|]; \\
  |u^{\Delta}(t,x)-u^{\Delta}(s,y)| & \leq C_{\phi}[|x-y|+M_Y^1(|s-t|+\Delta)].
\end{align*}
Thus, from Lemma \ref{estimate_final_interval}, the error bound between $u$ and $u^{\Delta}$ in the interval $[0,\Delta]$ can be refined as
\begin{equation}\label{estimate_final_interval_2}
\sup_{\bar{Q}_{T}\backslash {Q}_{T}^\Delta}|u-u^{\Delta}|\leq C{\Delta}.
\end{equation}

\emph{Step 2}. Next, we refine the consistency error estimates. From Proposition \ref{operator}, since $X=0$, term (I) disappears and $(II)\le\frac12\Delta^2|D^2\psi|_0 M^2_Y$. Thus,  $\mathcal{E}(\Delta,\psi)\leq \ C{\Delta}
|D^2\psi|_0.$ Plugging it into Proposition \ref{scheme property}(iii) yields
\begin{equation}\label{consistant_error22} |\partial_t\psi-G(D_{x}\psi,0)-{S}(\Delta,x,\psi,\psi(t-\Delta,\cdot))|
 \le
 C\Delta\left(|D^2_x\psi|_0+|\partial_{t}^2\psi|_{0}+|\partial_{t}D_x\psi|_{0}\right).
 \end{equation}

\emph{Step 3}. We modify the mollifiers $\rho_{\varepsilon}$ in (\ref{mollifer}) by
\begin{equation}\label{mollifier2}
\rho_{\varepsilon}(t,x):=\frac{1}{\varepsilon^{1+d}}\rho\left(\frac{t}{\varepsilon},\frac{x}{\varepsilon}\right),
\end{equation}
and redefine $u_\varepsilon$ as
$$u_{\varepsilon}(t,x)=u*\rho_{\varepsilon}(t,x)=\int_{-\varepsilon< \tau< 0}\int_{|e|<
\varepsilon}u(t-\tau,x-e)\rho_{\varepsilon}(\tau,e)ded\tau.$$
Thus, the regularity of $u$ implies that $|u-u^{\varepsilon}|\leq C\varepsilon$, and $|\partial_{t}^iD_{x}^ju_{\varepsilon}|_0\leq C\varepsilon^{1-i-|j|}.$

\emph{Step 4}. Substituting $u_{\varepsilon}$ with $\psi$ in the consistency error estimate (\ref{consistant_error22}) and using the fact that $\partial_tu_{\varepsilon}-G(D_{x}u_{\varepsilon},0)\geq 0$ yield   $${S}(\Delta,x,u_{\varepsilon}(t,x),u_{\varepsilon}(t-\Delta,\cdot))\ge -C\Delta\varepsilon^{-1}.$$ Furthermore, choosing $\varepsilon=\Delta^{1/2}$ and following along the similar arguments as in the proof for Theorem \ref{theorem_error_1}, we obtain that
$$u^{\Delta}-u\leq \sup_{\textcolor[rgb]{1.00,0.00,0.00}{\bar{Q}_{T}\backslash {Q}_{T}^\Delta}}(u-u^{\Delta})^{+}+C\Delta^{1/2}\le C\Delta^{1/2},$$
where we used (\ref{estimate_final_interval_2}) in the last inequality.

\emph{Step 5}. To prove the other side inequality, we mollify $u^{\Delta}$ with $\rho_{\varepsilon}(t,x)$ given in (\ref{mollifier2}), i.e. $u_{\varepsilon}^\Delta(t,x)=u^\Delta*
\rho_{\varepsilon}(t,x)$. Then, the regularity of $u^{\Delta}$ implies that
\begin{equation*}
|u^\Delta-u^\Delta_{\varepsilon}|_0\leq C(\varepsilon+\Delta),
\end{equation*}
and
\begin{equation*}
|\partial_{t}^iD_{x}^ju^\Delta_{\varepsilon}|_0\leq C\varepsilon^{-i-|j|}(\varepsilon+\Delta),
\end{equation*}

\emph{Step 6} Substituting $u^{\Delta}_{\varepsilon}$ with $\psi$ in the consistency error estimate (\ref{consistant_error22}) and using the fact that ${S}(\Delta,x,u^{\Delta}_{\varepsilon}(t,x),u^{\Delta}_{\varepsilon}(t-\Delta,\cdot))\geq 0$ yield
\begin{align*}
&\ \partial_{t}u_{\varepsilon}^\Delta-G(D_{x}u_{\varepsilon}^\Delta,0)
\ge -C(\varepsilon+\Delta)\Delta\varepsilon^{-2},
\end{align*}
In turn, $u-u_{\varepsilon}^\Delta\leq C(\varepsilon+\Delta)(1+\Delta\varepsilon^{-2})$, and by choosing $\varepsilon=\Delta^{1/2}$ and following along the similar arguments as in the proof for Theorem \ref{theorem_error_111}, we obtain
$$u-u^{\Delta}=(u-u_{\varepsilon}^\Delta)+(u_{\varepsilon}^\Delta-u^{\Delta})\le C\Delta^{1/2},$$
which is the desired convergence rate.

\subsection{Central limit theorem: Comparison with \cite{Krylov4} and \cite{Song}}
Assume that $Y=0$, we obtain the central limit theorem as in \cite{Krylov4} and \cite{Song}, but with an improved convergence rate. To this end, choosing
$\Delta=1/n$, by the representation formula (\ref{clt_representation}), we have $$u^{1/n}(1,0)=\sube[\phi(\sum_{i=1}^n \frac{X_i}{\sqrt{n}})],$$ where $X_1=X$, $X_{i+1} \stackrel{d}{=} X_i$ and $X_{i+1}$ is independent of $(X_1,...,X_i)$ for each $i=1,...,n-1$. On the other hand, since $G(p,A)=\sube[\frac12\langle AX,X\rangle]$ and $\zeta=0$, by (\ref{solution}), we have $$u(1,0)=\widetilde{\mathbb{E}}[\phi(\xi)]=\mathcal{N}_G(\phi),$$ where $\mathcal{N}_G$ denotes the corresponding G-normal distribution. Under Assumption \ref{assumption1}, Theorem \ref{Theorem_CLT} then yields the following \emph{central limit theorem} in the degenerate case
\begin{equation}\label{CLT_1}
\left|\sube[\phi(\sum_{i=1}^n \frac{X_i}{\sqrt{n}})]-\mathcal{N}_G(\phi)\right|\leq Cn^{-\beta/6}.
\end{equation}
Moreover, if the second moment of the random vector $X$ is non-degenerate, i.e. $\underline{\sigma}^2:=-\sube[-|X|^2]>0,$
and the initial data $\phi\in\mathcal{C}_b^1(\mathbb{R}^d)$, i.e. $\phi$ is bounded and Lipschitz continuous, then
\begin{equation}\label{CLT_2}
\left|\sube[\phi(\sum_{i=1}^n \frac{X_i}{\sqrt{n}})]-\mathcal{N}_G(\phi)\right|\leq Cn^{-\max\{\frac{\alpha}{2},\frac16\}}.
\end{equation}

Note that the above convergence rate in (\ref{CLT_1}) for the degenerate case improves Theorem 1.1 of Krylov \cite{Krylov4}, where the author considers a one-dimensional stochastic control problem and obtains the convergence rate $\frac{\beta^2}{4+2\beta}(\leq \frac{\beta}{6})$. Moreover, the convergence rate in (\ref{CLT_2}) for the non-degeneate case improves Theorem 4.5 of Song \cite{Song}, where the author obtains the convergence rate $\frac{\alpha}{2}$.

\subsection{\textcolor[rgb]{1.00,0.00,0.00}{Linear central limit theorem}}

We conclude the paper with a brief discussion of the convergence rate of the linear central limit theorem via a monotone scheme. First, we consider the smooth test function $\phi \in \mathcal{C}_b^{2+1}(\mathbb{R})$, meaning it is twice differentiable, with its second-order derivative being uniformly Lipschitz continuous. In this case, we obtain the following central limit theorem. The proof is similar to the proof of Theorem \ref{Theorem_CLT} and is omitted.

\begin{corollary}
Let $\{X_i\}_{i\geq 1}$ be a sequence of real-valued i.i.d. random variables with finite third moment $M_X^3<\infty$ and $S_n=\sum_{{i=1}}^{n}\frac{X_i}{\sqrt{n}}$. Then, for $\phi\in \mathcal{C}_b^{2+1}(\mathbb{R})$,
\begin{equation}\label{smooth_case}
\left|\mathbb{E}\left[\phi\left(S_n\right)\right]-\mathbb{E}[\phi(\xi)]\right| \leq\left(1+M_X^3\right)\left[\phi_{xx}\right]_{\mathcal{C}^1} n^{-\frac{1}{2}},
\end{equation}
where $\xi$ follows standard normal distribution.
\end{corollary}

However, the linear central limit theorem requires the test function $\phi$ to a step function, denoted as $\phi(\cdot) = \mathbf{1}_{{\cdot \leq x}}$ parameterized by $x \in \mathbb{R}$. Therefore, the above convergence rate result does not apply directly to such a test function $\phi(\cdot)$. To obtain the convergence rate for this test function $\phi(\cdot)$, it is noteworthy that, by integration by parts, we have
\begin{align}\label{IBP}
& \int \phi_x(x)\left[G(x)-F_n(x)\right] d x \notag\\
= & \int \phi(x) d F_n(x)-\int \phi(x) d G(x) \notag\\
= &\ \mathbb{E}\left[\phi\left({S}_n\right)\right]-\mathbb{E}[\phi(\xi)],
\end{align}
where $F_n(x)=\mathbb{E}[\mathbf{1}_{\{S_n\leq x\}}]$ and $G(x)=\mathbb{E}[\mathbf{1}_{\{\xi\leq x\}}]$.

\begin{proposition}
Let $\{X_i\}_{i\geq 1}$ be a sequence of real-valued i.i.d. random variables with finite third moment $M_X^3<\infty$ and $S_n=\sum_{{i=1}}^{n}\frac{X_i}{\sqrt{n}}$. Then,
$$
\left|F_n-G\right|_0 \leq\left(M_X^3+5\right) n^{-\frac{1}{8}} .
$$
\end{proposition}

\begin{proof} Let $h: \mathbb{R} \rightarrow \mathbb{R}$ be such that $h(0)=0$ and
$$
h_x(x)=x^2 \mathbf{1}_{\{0 \leq x \leq 1\}}+\left(2-(x-2)^2\right) \mathbf{1}_{\{1<x \leq 3\}}+(x-4)^2 \mathbf{1}_{\{3<x \leq 4\}}.
$$
Then, it is clear that $h \in \mathcal{C}_b^{2+1}(\mathbb{R})$ with $\left[h_{xx}\right]_{\mathcal{C}^1}=2$. Define for any $a \in \mathbb{R}$ and $\varepsilon>0$,
$$
\phi_{a, \varepsilon}(x):=h\left(\frac{x-a}{\varepsilon}\right).
$$
We then have $\phi_{a, \varepsilon} \in \mathcal{C}_b^{2+1}(\mathbb{R})$ with $\left[(\phi_{a, \varepsilon})_{xx}\right]_{\mathcal{C}^1}=2 \varepsilon^{-3}$. Applying $\phi_{a, \varepsilon}$ to (\ref{smooth_case}) and (\ref{IBP}) yields
\begin{equation}\label{1st_estimate}
\left|\int (\phi_{a, \varepsilon})_x(x)\left[G(x)-F_n(x)\right] d x\right| \leq 2\left(M_X^3+1\right) \varepsilon^{-3} n^{-\frac{1}{2}} .
\end{equation}

Focussing on the integral of the above left hand side, we further have that by
change of variable $y=\frac{x-a}{\varepsilon}$,
\begin{align}\label{2nd_estimate}
& \int (\phi_{a, \varepsilon})_x(x)\left[G(x)-F_n(x)\right] d x \notag\\
= & \int h_x(y)\left[G(a+\varepsilon y)-F_n(a+\varepsilon y)\right] d y \notag\\
\leq & \left[G(a)-F_n(a)+4 \varepsilon\right] \int_0^4 h_x(y) d y \notag\\
= &\ 4\left[G(a)-F_n(a)+4 \varepsilon\right],
\end{align}
where we use that $h_x \geq 0,[G]_{\mathcal{C}^1} \leq 1$ and that $F_n$ is monotone. Combining (\ref{1st_estimate}) and (\ref{2nd_estimate}), we have
$$
F_n(a)-G(a) \leq \frac{1}{2}\left(M_X^3+1\right) \varepsilon^{-3} n^{-\frac{1}{2}}+4 \varepsilon \leq\left(M_X^3+5\right) n^{-\frac{1}{8}}
$$
by letting $\varepsilon=n^{-1 / 8}$. Similarly,
$$
\begin{aligned}
& \int h_x(y)\left[G(a+\varepsilon y)-F_n(a+\varepsilon y)\right] d y \\
\geq & \left[G(a+4 \varepsilon)-F_n(a+4 \varepsilon)-4 \varepsilon\right] \int_0^4 h_x(y) d y \\
= &\ 4\left[G(a+4 \varepsilon)-F_n(a+4 \varepsilon)-4 \varepsilon\right],
\end{aligned}
$$
and thus
$$
G\left(a+4 n^{-1 / 8}\right)-F_n\left(a+4 n^{-1 / 8}\right) \leq\left(M_X^3+5\right) n^{-\frac{1}{8}}.
$$
By the arbitrary of $a$, we conclude that
$$
\left|F_n-G\right|_0 \leq\left(M_X^3+5\right) n^{-\frac{1}{8}} .
$$
\end{proof}

Although the rate of $1/8$ we established is slower than that of $1/2$ in the Berry-Esseen theorem due to the introduction of an approximation to step functions by $\mathcal{C}_b^{2+1}$ functions ($\phi_{a, \varepsilon}$ in the proof above), we have demonstrated a much simpler way to establish a rate of Berry-Esseen type convergence. In fact, this also addresses the challenge encountered by Stein's method when applied to establish a lower convergence rate than $1/2$. Establishing the convergence rate of $1/2$ for the linear central limit theorem is more challenging and will be left for future research.

\begin{appendix}
\section{Proof of Proposition \ref{BSB_theorem}}
The convergence rate $\frac{\alpha}{2}$ follows immediately from Theorem \ref{maintheorem}(iii). To establish the other convergence rate $\frac{1}{4}$, we only need to prove the following consistency error estimate
\begin{align}\label{operator_E_estimate}
\mathcal{E}(\Delta,\psi)&:=
\left|\frac{\mathbf{S}(\Delta)\psi-\psi}{\Delta}-G(D\psi,D^2\psi)\right|_{0}\notag\\
&\leq C\Delta(|D^4\psi|_0+|D^3\psi|_0+|D^2\psi|_0),
\end{align}
for any test function $\psi\in\mathcal{C}_b^{\infty}(\mathbb{R})$.
Note that (\ref{operator_E_estimate}) is a refinement of the consistency error estimate in Proposition \ref{operator}.
The rest of the proof then follows along a similar argument and procedure used in the proof of Theorem \ref{maintheorem}. To establish (\ref{operator_E_estimate}), with $Y=r-\frac12X^2$, we have
\begin{align*}
 &\ \mathbf{S}(\Delta)\psi(x)-\psi(x)-\Delta G(D\psi(x),D^2\psi(x)) \\
 \le &\
 \sube[\psi(x+\sqrt{\Delta}X+\Delta Y)-\psi(x)-\Delta D\psi(x)Y-\frac12 \Delta D^2\psi(x)X^2] \\
 \le &\
 \sube[\psi(x+\sqrt{\Delta}X)-\psi(x)-\frac12  D^2\psi(x)\Delta X^2]\\
 &\ +
 \sube[\psi(x+\sqrt{\Delta}X+\Delta Y)-\psi(x+\sqrt{\Delta}X)- D\psi(x+\sqrt{\Delta}X)\Delta Y]\\
 &\ +\sube[D\psi(x+\sqrt{\Delta}X)\Delta Y- D\psi(x)\Delta Y]:=(I)+(II)+(III).
\end{align*}

Since $\sube[X]=\sube[-X]=\sube[X^3]=\sube[-X^3]=0$, we have
\begin{align*}
(I) = &\
\sube\left[\sqrt{\Delta}D\psi(x)X+\int_x^{x+\sqrt{\Delta}X}\int_x^s(D^2\psi(u)-D^2\psi(x))duds\right] \\
= &\
\sube\left[\int_x^{x+\sqrt{\Delta}X}\int_x^s\int_x^uD^3\psi(p)dpduds\right] \\
= &\
\sube\left[\textcolor[rgb]{1.00,0.00,0.00}{\frac16\Delta^{3/2}D^3\psi(x)X^3}+\int_x^{x+\sqrt{\Delta}X}\int_x^s\int_x^u(D^3\psi(p)-D^3\psi(x))dpduds\right] \\
= &\
\sube\left[\int_x^{x+\sqrt{\Delta}X}\int_x^s\int_x^u(D^3\psi(p)-D^3\psi(x))dpduds\right] \\
\le &\
|D^4\psi|_0\sube\left|\int_x^{x+\sqrt{\Delta}X}\int_x^s\int_x^u|p-x|dpduds\right| \le \Delta^2 |D^4\psi|_0 M_X^4.
\end{align*}
Likewise, Taylor's expansion yields that
\begin{align*}
(II)=&\ \sube\left[\int_0^1(1-s)D^2\psi(x+\sqrt{\Delta}X+s\Delta Y)\Delta^2Y^2ds\right]\\
\leq &\
\sube\left[\int_0^1(1-s)ds|D^2\psi|_0\Delta^2Y^2\right]\leq \frac12\Delta^2|D^2\psi|_0M_Y^2,
\end{align*}
and
\begin{align*}
(III)=&\ \sube\left[\int_0^1D^2\psi(x+s\sqrt{\Delta}X)\sqrt{\Delta}Xds\Delta Y\right]\\
= &\
\sube\left[\int_0^1[D^2\psi(x+s\sqrt{\Delta}X)-D^2\psi(x)]ds\Delta^{\frac32} X Y+D^{2}\psi(x)\Delta^{\frac32} X Y\right]\\
\leq &\
\sube\left[\int_0^1[D^2\psi(x+s\sqrt{\Delta}X)-D^2\psi(x)]ds\Delta^{\frac32} X Y\right]+\sube\left[D^{2}\psi(x)\Delta^{\frac32} X Y\right]\\
=&\
\sube\left[\int_0^1\int_0^1D^3\psi(x+us\sqrt{\Delta}X)s\sqrt{\Delta}Xduds\Delta^{\frac32} X Y\right]\\
\leq&\
\sube\left[\int_0^1\int_0^1sduds|D^3\psi|_0\Delta^2|X|^2|Y|\right]\leq \frac14\Delta^2|D^3\psi|_0(M_X^4+M_Y^2),
\end{align*}
where we also used the fact that $\sube[XY]=\sube[-XY]=0$. The consistency error estimate (\ref{operator_E_estimate}) then follows by combining (I)-(III).
\end{appendix}



\end{document}